\documentclass[a4paper,10pt]{article}
\usepackage{amsthm,amssymb,amsfonts,mathrsfs}
\usepackage{xcolor}
\usepackage{tcolorbox}
\usepackage{amsmath}
\usepackage{amssymb}
\usepackage{multicol} % İki sütunlu düzen için
\usepackage{graphicx}
\usepackage[numbers]{natbib}
\usepackage{xcolor}
\usepackage{titling}
 \usepackage{caption} % preamble'da ekle
\newtheorem{theorem}{Theorem}[section]
\usepackage{tabularx}  
\usepackage{array}
\usepackage{multirow}
\usepackage{pdflscape} % O \usepackage{lscape}
\usepackage{makecell}
\usepackage{hyperref}
\hypersetup{colorlinks=true, linkcolor=blue}

\usepackage[
hmarginratio=1:1,
a4paper,
bottom=3.4cm,
top=3.6cm,
left=2.7cm,
right=2.7cm,
headheight=16pt
]{geometry}

\newtheorem{proposition}[theorem]{Proposition}%
\newtheorem{lemma}[theorem]{Lemma}
\newtheorem{corollary}[theorem]{Corollary}

\newtheorem{example}[theorem]{Example}%
\newtheorem{remark}[theorem]{Remark}%
\newcommand{\ad}{\operatorname{ad}}
\newcommand{\im}{\operatorname{Im}}

\newtheorem{definition}[theorem]{Definition}%

\title{Transversely K\"ahler almost contact metric Lie algebras}

\author{Giulia Dileo, Deniz Poyraz, Bayram \c{S}ahin
}
\date{}

\renewcommand{\maketitlehookd}{\begin{abstract}

We study transversely K\"ahler almost contact metric Lie algebras $(\mathfrak{g},\varphi,\xi,\eta,g)$ such that the structure $1$-form $\eta$ is a contact form. They include both quasi Sasakian and anti-quasi-Sasakian Lie algebras of maximal rank. In the case where the center of the Lie algebra is nontrivial, they are $1$-dimensional central extensions of K\"ahler Lie algebras via a symplectic form. 

We investigate the $5$-dimensional case, obtaining a classification of $\eta$-Einstein transversely K\"ahler almost contact metric Lie algebras of maximal rank. If the center is trivial, the structure is always $\alpha$-Sasakian. If the center is nontrivial and the K\"ahler quotient $\mathfrak{g}/\mathfrak{z(g)}$ is not abelian, the structure is quasi Sasakian; it is $\alpha$-Sasakian  on central extensions of K\"ahler-Einstein $4$-dimensional Lie algebras, and not conversely. Up to isomorphisms, the Heisenberg Lie algebra $\mathfrak{h}_5$ is the only $5$-dimensional Lie algebra admitting $\eta$-Einstein transversely K\"ahler structures which are not quasi Sasakian, including anti-quasi-Sasakian structures. In fact, we show that any $5$-dimensional anti-quasi-Sasakian Lie algebra is isomorphic to $\mathfrak{h}_5$.
\end{abstract}

\medskip
\textbf{MSC (2020):} 53C25, 53C15, 53D15.\smallskip

\textbf{Keywords:} almost contact metric Lie algebras, K\"ahler Lie algebras, quasi Sasakian structures,  anti-quasi-Sasakian structures, $\alpha$-Sasakian structures, $\eta$-Einstein metric}

\begin{document}
\maketitle 
\section*{Introduction} 
 
Almost contact metric manifolds are a fundamental class of odd-dimensional Riemannian manifolds $M^{2n+1}$, whose structural group is reducible to $U(n)\times 1$. They are characterized by a geometric structure $(\varphi,\xi,\eta,g)$, where $\varphi$ is a (1,1)-tensor field, $\xi$ a vector field, $\eta$ a 1-form, and $g$ a compatible Riemannian metric (\cite{Blair,S1,Y}). The presence of the distinguished direction given by the Reeb vector field $\xi$, makes it natural to investigate the transverse geometry with respect to the orbits $\xi$. Transversely K\"ahler almost contact metric manifolds are manifolds for which the structure $(\varphi,g)$ is projectable along the orbits of $\xi$ and the transverse geometry is K\"ahler (\cite{DG,C}). They include, for instance, quasi Sasakian manifolds, and their subclass of $\alpha$-Sasakian manifolds, Sasakian when $\alpha=1$, all well-known classes, extensively studied in the literature (\cite{S2, BG, BG1, Blair-qS, Kanemaki1}). For a quasi Sasakian manifold, considering a local Riemannian submersion $U\to U/\xi$, the K\"ahler base space turns out to be endowed with a closed $2$-form $\omega$ of type $(1,1)$, given by the projection of the $2$-form $d\eta$. In general $\omega$ can be degenerate. If $\eta$ has maximal rank, namely $\eta$ is a contact form, $\omega$ is nondegenerate. This is the case of $\alpha$-Sasakian manifolds, for which $d\eta$ is a multiple of the fundamental $2$-form $\Phi$ associated to the structure $(\varphi,\xi,\eta,g)$, so that $\omega$ is a multiple of the K\"ahler form of $U/\xi$. Recently, a new class of transversely K\"ahler almost contact metric manifolds has been introduced and investigated in \cite{DG, Dario}. These are the anti-quasi-Sasakian manifolds, for which the $2$-form $d\eta$ projects onto a closed $2$-form $\omega$ of type $(2,0)$ on the K\"ahler space of orbits $U/\xi$. In this case, if $\eta$ is a contact form, the K\"ahler space $U/\xi$ is endowed with a complex symplectic structure, which forces the dimension of $M$ to be of type $4n+1$.

Almost contact (metric) Lie algebras have been extensively studied, providing a framework for understanding left invariant geometric structures on Lie groups (\cite{Diatta,Diatta2,AF,CF,LL}). Within this framework, Sasakian Lie algebras have attracted particular attention as a distinguished subclass endowed with rich geometric and algebraic properties, and have been  studied in terms of their structure, curvature behavior, and classification. 
The study of Sasakian Lie algebras began with the observation that these algebras naturally split into two main types: those with a $1$-dimensional center generated by the Reeb vector, and those with a trivial center (\cite{AF}). In particular, in the nontrivial center case, the Lie algebras can be described as central extensions of K\"ahler Lie algebras via the cocycle defined by the K\"ahler form. In \cite{AF}, the authors classify all $5$-dimensional Sasakian Lie algebras. The larger class of $K$-contact Lie algebras is investigated in \cite{CF}, which also contains a full classification of the $5$-dimensional case.

In this work we investigate the class of transversely K\"ahler almost contact metric Lie algebras $(\mathfrak{g},\varphi,\xi,\eta,g)$, with particular attention to the case where the structure has maximal rank, i.e. $\eta\wedge(d\eta)^n\ne0$. Even in this case, the center of the Lie algebra is either trivial, or $1$-dimensional and generated by the Reeb vector. If the center is nontrivial, transversely K\"ahler almost contact metric Lie algebras $(\mathfrak{g},\varphi,\xi,\eta,g)$ of maximal rank are in one-to-one correspondence with K\"ahler Lie algebras $(\mathfrak{h},J,h,\omega)$ endowed with a symplectic 2-form $\omega$ (Theorem~\ref{theorem6}). The structure $(\varphi, \xi,\eta,g)$ is quasi Sasakian or anti-quasi-Sasakian if and only if the $2$-form $\omega$ is $J$-invariant or $J$-anti-invariant, respectively. In this framework, we investigate the curvature relations between  the K\"ahler Lie algebra $(\mathfrak{h},J,h)$ and the almost contact metric Lie algebra $(\mathfrak{g}, \varphi, \xi,\eta,g)$ obtained as central extension of $\mathfrak{h}$ via a symplectic form $\omega$ (Section \ref{section-curvature}). 

We examine with special attention $5$-dimensional Lie algebras, focusing on the $\eta$-Einstein condition. This means that the Ricci curvature tensor satisfies  \begin{equation}\label{eta-einstein}\mathrm{Ric}=\lambda g+\mu\eta\otimes\eta
\end{equation}
for $\lambda,\mu\in\mathbb{R}$. This is a weaker condition than the Einstein condition, corresponding to $\mu=0$, which can be quite restrictive for almost contact metric structures. For instance, the only simply connected Sasaki-Einstein homogeneous $5$-manifolds are $S^5$ and $S^2\times S^3$. This is consequence the classification of compact, simply connected, homogeneous contact $5$-manifolds in \cite{PV}; see also \cite[Chapter 11]{BG} for  Sasaki-Einstein structures on $S^2\times S^3$. The $\eta$-Einstein condition for Sasakian structures is equivalent to require K\"ahler-Einstein transverse geometry with respect to the orbits of $\xi$. This geometry has been investigated in \cite{BG1} and recently in \cite{AC} in the context of Sasakian Lie algebras.

In this paper, we first classify $5$-dimensional anti-quasi-Sasakian Lie algebras, showing that they are isomorphic to the Heisenberg Lie algebra $\mathfrak{h}_5$ (Theorem~\ref{theorem3}). This has been recently proved for nilpotent Lie algebras in \cite{Dario}. 
Here we obtain the result using the fact that any $5$-dimensional anti-quasi-Sasakian Lie algebra admits a contact Calabi-Yau structure, as defined in \cite{TV}. In \cite{AF2} the authors prove that, up to isomorphisms, $\mathfrak{h}_5$ is the only $5$-dimensional contact Calabi-Yau Lie algebra with $[\mathfrak{g},\mathfrak{g}]\ne \mathfrak{g}$. We show that this assumption can be actually removed, owing to the fact that any contact Calabi-Yau Lie algebra carries a Sasakian structure, and one takes into account the classification of $5$-dimensional Sasakian Lie algebras in \cite{AF}. It is already known that the anti-quasi-Sasakian structure on $\mathfrak{h}_5$ is $\eta$-Einstein, more precisely null $\eta$-Einstein, in the sense that it is transversely Ricci flat (\cite{DG}).

Considering the more general setting of $5$-dimensional transversely K\"ahler almost contact metric Lie algebras, we treat separately the cases of trivial and nontrivial center. In the case of trivial center, we show that every transversely K\"ahler Lie algebra of maximal rank is necessarily isomorphic to $\mathfrak{aff}(\mathbb{R}) \times \mathfrak{sl}(2,\mathbb{R})$, or $\mathfrak{aff}(\mathbb{R}) \times \mathfrak{su}(2)$, or to a non-unimodular solvable Lie algebra $\mathfrak{g}_0\simeq \mathbb{R}^2 \ltimes\mathfrak{h}_3$. Furthermore, the structure is necessarily $\alpha$-Sasakian (Theorem \ref{theorem13}). In fact, these are, up to isomorphisms, exactly all the Lie algebras with trivial center admitting a Sasakian structure, according to \cite{AF}. Among them, the only ones admitting $\eta$-Einstein structures are  $\mathfrak{sl}(2, \mathbb{R}) \times \mathfrak{aff}(\mathbb{R})$ and $\mathfrak{g}_0$ (Corollary \ref{corollary-eta-eistein}).

In the case of nontrivial center, all $5$-dimensional transversely K\"ahler almost contact metric Lie algebras can be obtained as $1$-dimensional central extensions of $4$-dimensional K\"ahler Lie algebras, via a symplectic form. Then, we use Ovando's classifications of both $4$-dimensional K\"ahler and symplectic Lie algebras (\cite{Ovandop} and \cite{Ovando}, respectively) and we classify $\eta$-Einstein transversely K\"ahler almost contact metric Lie algebras of maximal rank and with nontrivial center, as in Table \ref{table-eta-einstein} (Theorem \ref{theorem14}). In particular, in all the cases where the K\"ahler quotient $\mathfrak{g}/\mathfrak{z(g)}$ is not abelian, the almost contact metric structure is quasi Sasakian. All, and only all, $4$-dimensional K\"ahler-Einstein Lie algebras admit $\alpha$-Sasakian $\eta$-Einstein extensions; among them $\mathfrak{rr}'_{3,0}$ is the only non abelian K\"ahler-Einstein Lie algebras that also admits $\eta$-Einstein extensions which are quasi Sasakian non $\alpha$-Sasakian. The Heisenberg Lie algebra $\mathfrak{h}_5$, obtained as central extension of the abelian Lie algebra, is the only $5$-dimensional Lie algebra admitting $\eta$-Einstein structures which are not quasi Sasakian, including the already mentioned anti-quasi-Sasakian structures. In any case the $\eta$-Einstein structure is not Einstein. 

As a consequence, putting together the results regarding trivial and nontrivial center, we complete the classification of $5$-dimensional $\eta$-Einstein Sasakian Lie algebras in \cite{AF}, and more generally, we classify $5$-dimensional $\eta$-Einstein quasi Sasakian Lie algebras of maximal rank.

\medskip

\subsubsection*{Acknowledgements and declarations.} 
This work was partially carried out during the visit of Deniz Poyraz at the Department of Mathematics of the University of Bari  under the fellowship TÜBİTAK 2214-A, as part of her PhD programme.

Deniz Poyraz would like to express her sincere gratitude to Professor Dileo for her kind hospitality during her stay in Bari.

Giulia Dileo was partially supported by the National Center of HPC, Big Data and Quantum Computing, MUR CN00000013, CUP H93C22000450007, Spoke 10, and by PRIN 2022MWPMAB - \lq\lq Interactions between Geometric Structures and Function Theories\rq\rq. Giulia Dileo is member of INdAM - GNSAGA (Gruppo Nazionale per le Strutture Algebriche, Geometriche e le loro Applicazioni).

\section{Preliminaries}

In this section, we recall fundamental notions and structures on Lie algebras that will be used throughout the paper. Whenever one considers a Lie group $G$ with Lie algebra $\mathfrak{g}$, a structure defined by special tensors on $\mathfrak{g}$ determines in a natural manner a left invariant analogous geometric structure on the Lie group $G$. Since we will be mainly concerned with Lie algebras, we recall the notions on Lie algebras.\medskip

On a Lie algebra $\mathfrak{g}$ of even dimension $2n$, a \emph{symplectic structure} is given by a closed 2-form $\omega \in \Lambda^2(\mathfrak{g}^*)$ satisfying $\omega^n \neq 0$, namely, $\omega$ is nondegenerate, or of maximal rank.

 A \emph{complex structure} on a $2n$-dimensional real Lie algebra $\mathfrak{g}$ is given by an endomorphism $J:\mathfrak{g}\to\mathfrak{g}$ such that $J^2=-I$ and $[J,J]=0$, where $[J,J]$ is the Nijenhuis tensor defined by
\[[J,J](X,Y)=[JX,JY]-[X,Y]-J[JX,Y]-J[X,JY]\]
for any $X,Y\in\mathfrak{g}$.
A pair $(J,g)$ is said to define a \emph{K\"ahler structure} on $\mathfrak{g}$ whenever $J$ is a complex structure and $g$ is a positive-definite inner product compatible with $J$, that is, $g(JX,JY)=g(X,Y)$ for all $X,Y\in\mathfrak{g}$, and such that the associated 2-form $\Omega(X,Y)=g(X,JY)$ is closed and nondegenerate.
In this case, $(\mathfrak{g},J,g)$ is called a \emph{K\"ahler Lie algebra}.
\medskip

Considering now an odd-dimensional Lie algebra $\mathfrak{g}$, if $\dim\mathfrak{g}=2n+1$, a \emph{contact structure} on $\mathfrak{g}$ is a $1$-form $\eta\in\mathfrak{g}^*$ such that $\eta\wedge (d\eta)^n\ne0$. The \emph{Reeb vector} associated to the contact structure $\eta$ is the unique vector $\xi\in\mathfrak{g}$ such that $\eta(\xi)=1$ and $d\eta(\xi,\cdot)=0$.

 An \emph{almost contact metric Lie algebra} is a $(2n+1)$-dimensional Lie algebra $(\mathfrak{g} ,\varphi,\xi,\eta,g)$ where $\varphi\in End(\mathfrak{g})$, $\xi\in\mathfrak{g}$ (usually also called Reeb
vector), $\eta\in \mathfrak{g}^*$, and $g$ is a positive-definite inner product on $\mathfrak{g}$ such that
\[
   \varphi^2 = -I+\eta \otimes\xi, \qquad \eta(\xi)=1,\] %\label{d1.1}
  \[ g(\varphi X,\varphi Y)=g(X,Y)-\eta(X)\eta(Y)\qquad \forall X,Y \in \mathfrak{g}.\]
It follows that ${\varphi(\xi)=0}$, ${\eta \circ \varphi =0}$,
$\eta=g(\xi,\cdot)$ and $\|\xi\|=1$. The Lie algebra $\mathfrak{g}$ naturally splits into the direct sum of a horizontal and a vertical subspace:
\[
\mathfrak{g} = \mathfrak{h} \oplus \langle \xi \rangle,
\]
where the \emph{horizontal subspace} is defined by 
\(\mathfrak{h} := \ker\eta = \im \varphi\), and for which the following properties hold:
\[
\mathfrak{h} = \langle \xi \rangle^\perp, 
\qquad
\varphi^2|_\mathfrak{h} = -I,
\qquad
g(\varphi X, \varphi Y) = g(X,Y) \quad \forall X,Y \in \mathfrak{h}.
\]
The  fundamental $2$-form associated to an almost contact metric structure is defined by 
$${\Phi(X,Y)=g(X,\varphi Y)}.$$ If $d\eta=2\Phi$, then the almost contact metric structure is called a \emph{contact metric structure},  and in this case $\eta$ is a contact form. If, furthermore, the tensor $\mathcal{L}_{\xi}\varphi=\operatorname{ad}_\xi\circ\varphi-\varphi\circ\operatorname{ad}_\xi$ vanishes, the contact metric Lie algebra is said to be \emph{$K$-contact}. On a Lie group $G$ with Lie algrebra $\mathfrak{g}$ and left invariant contact metric structure $(\varphi,\xi,\eta,g)$, the vanishing of the Lie derivative $\mathcal{L}_{\xi}\varphi$ is equivalent to the fact that $\xi$ is a Killing vector field. Throughout the paper, for a tensor $t$ on an almost contact metric Lie algebra $(\mathfrak{g},\varphi,\xi,\eta,g)$, we will denote by $\mathcal{L}_\xi t$ the operator given by the evaluation at the identity of the Lie derivative of the corresponding left invariant tensor field defined on a Lie group $G$ with Lie algebra $\mathfrak{g}$.
\medskip

An almost contact metric Lie algebra is said to be:
\begin{itemize}
    \item \emph{cok\"ahler} if $N_\varphi = 0$, $d\eta = 0$, and $d\Phi = 0$,
    \item \emph{$\alpha$-Sasakian} if $N_\varphi = 0$ and $d\eta = 2\alpha\Phi$, $\alpha\in\mathbb{R}^*$, which is \emph{Sasakian} for $\alpha=1$,
    \item \emph{quasi Sasakian} if $N_\varphi = 0$ and $d\Phi = 0$,
\end{itemize}
 where $N_{\varphi}:=[\varphi,\varphi]+d\eta\otimes \xi$, which can be explicitly written as
 \begin{eqnarray*}%\label{denk1.1}
N_{\varphi} (X,Y) = [ \varphi X,\varphi Y]+\varphi ^2 [X,Y]-\varphi [X,\varphi Y]-\varphi [\varphi X,Y]+d\eta(X,Y)\xi
\end{eqnarray*}
 for any $X,Y\in \mathfrak{g}$. When $N_\varphi=0$, as in the above three cases, one says that the structure is \emph{normal}.
\medskip

Two almost contact metric Lie algebras $(\mathfrak{g}_1,\varphi_1,\xi_1,\eta_1,g_1)$ and $(\mathfrak{g}_2,\varphi_2,\xi_2,\eta_2,g_2)$ are said to be \emph{isomorphic} if there exists a Lie algebra isomorphism 
$f:\mathfrak{g}_1 \to \mathfrak{g}_2$ satisfying the following conditions:
\[
\begin{array}{l@{\hskip 1,2cm}l} 
     \text{(a)}\; f :(\mathfrak{g}_1,g_1)\rightarrow (\mathfrak{g}_2,g_2)~ \text{ is an isometry,} & \text{(b)}\; f(\xi_1)=\xi_2, \\
    \text{(c)}\; f \circ \varphi_1=\varphi_2\circ f, &\text{(d)}\;\eta_1=\eta_2 \circ f.
\end{array}
\]
In \cite{AF}, the authors classify $5$-dimensional Sasakian Lie algebras. 
Below, we summarize the main results obtained in their work:
\begin{theorem}\label{theo-Sasaki-nontrivial} \cite{AF} 
Any $5$-dimensional Sasakian Lie algebra $\mathfrak{g}$ with nontrivial center is isomorphic to one of the following Lie algebras: 
\begin{eqnarray*}
&&\mathfrak{g}_1=(0,0,0,0,e^{12}+e^{34})\simeq {\mathfrak{h}_5;}\\
&&\mathfrak{g}_2=(0,-e^{12},0,0,e^{12}+e^{34})\simeq {\mathfrak{aff}(\mathbb
R) \times \mathfrak{h}_3;}\\% \mathfrak{g}_2 \simeq \mathfrak{r}\mathfrak{r}_{3,0} not \eta Einstein
&&\mathfrak{g}_3=(0,-e^{13},e^{12},0,e^{14}+e^{23})\simeq {\mathbb{R}\ltimes(\mathfrak{h}_3 \times \mathbb{R});} \\ % \mathfrak{g}_3 \simeq \mathfrak{r}\mathfrak{r}'_{3,0} \eta Einstein
&&\mathfrak{g}_4=(0,-e^{12},0,-e^{34},e^{12}+e^{34})\simeq { \mathfrak{aff}(\mathbb{R}) \times \mathfrak{aff}(\mathbb{R}) \times \mathbb{R};} \\ % \mathfrak{g}_4 \simeq  \mathfrak{r}_2\mathfrak{r}_2
&&\mathfrak{g}_5=(\frac{1}{2}e^{14},\frac{1}{2}e^{24},-e^{12}+e^{34},0,e^{12}-e^{34})\simeq {\mathbb{R} \times (\mathbb{R}\ltimes\mathfrak{h}_3);} \\ % \mathfrak{g}_5 \simeq \mathfrak{d}_{4,1/2}
&&\mathfrak{g}_6=(2e^{14},-e^{24},-e^{12}+e^{34},0,e^{23})\simeq {\mathbb{R}\ltimes\mathfrak{n}_4;} \\ % \mathfrak{g}_6 \simeq \mathfrak{d}_{4,2}
&&{\mathfrak{g}^\delta_7= (\frac{\delta}{2}e^{14}+e^{24},-e^{14}+\frac{\delta}{2}e^{24},-e^{12}+\delta e^{34},0,e^{12}-\delta e^{34})\simeq {\mathbb{R} \times (\mathbb{R}\ltimes\mathfrak{h}_3), \delta>0;}} \\ % \mathfrak{g}^\delta_7 \simeq \mathfrak{d}'_{4,\delta}
&&\mathfrak{g}^\delta_8=(e^{14},\delta e^{34},-\delta e^{24},0,e^{14}+e^{23})\simeq {\mathbb{R}\ltimes_\delta (\mathfrak{h}_3 \times \mathbb{R}), \delta>0}.%\mathfrak{g}^\delta_8 \simeq \mathfrak{r}'_{4,0,\delta}
\end{eqnarray*}

 \end{theorem}
  The above notation for the Lie algebras is the same we adopt throughout the paper. Denoting by  ${\lbrace e_1,e_2,e_3,e_4,e_5 \rbrace}$ a basis of $\mathfrak{g}$ and by ${\lbrace e^1,e^2,e^3,e^4,e^5 \rbrace}$ its dual basis in $\mathfrak{g}^*$, the notation    ${\mathfrak{g}_2=(0,-e^{12},0,0,e^{12}+e^{34})}$ means that
\[
       de^1=de^3=de^4=0,\quad de^2=-e^{12},\quad 
       de^5=e^{12}+e^{34},  
 \]
where ${ e^{ij}:=e^i\wedge e^j}$ for all $i,j$.
 \begin{theorem}\label{theo-Sasaki-trivial} \cite{AF} 
If a $5$-dimensional Sasakian Lie algebra $\mathfrak{g}$ has trivial center, then it is isomorphic to one of the following Lie algebras: the direct products  $\mathfrak{sl}(2,\mathbb{R})\times \mathfrak{aff}(\mathbb{R})$, $\mathfrak{su}(2)\times \mathfrak{aff}(\mathbb{R})$, or a non-unimodular solvable Lie algebra $\mathfrak{g}_0 \simeq \mathbb{R}^2 \ltimes\mathfrak{h}_3$.
\end{theorem}

 \section{Anti-quasi-Sasakian Lie algebras}\label{sectionaqS}
In this section, we study $5$-dimensional anti-quasi-Sasakian Lie algebras, showing that they are always isomorphic to the Heisenberg Lie algebra $\mathfrak{h}_5$ (Theorem \ref{theorem3}). As already remarked in the Introduction, anti-quasi-Sasakian structures are a special class of transversely K\"ahler almost contact metric structures, introduced in \cite{DG}. As we will see in the next Sections, Theorem \ref{theorem3} is also consequence of the results concerning $5$-dimensional transversely K\"ahler Lie algebras (Theorems \ref{theorem6} and \ref{theorem13}, Remark \ref{remark aqS}). We present in this section a specific proof of this result, which directly makes use of the classification of $5$-dimensional Sasakian Lie algebras, in order to highlight the strict connection between anti-quasi-Sasakian, contact Calabi-Yau, and Sasakian structures.
\medskip

To proceed, we first recall the relevant definitions and concepts from \cite{DG}. \\

An almost contact metric Lie algebra $(\mathfrak{g}, \varphi, \xi, \eta, g)$ is said to be 
\emph{anti-quasi-Sasakian} (aqS for short) if
\[
d\Phi = 0, \qquad N_\varphi = 2\, d\eta \otimes \xi.
\]
The second equation, referred to as the \emph{anti-normal condition}, implies that 
$d\eta(\xi, \cdot) = 0$, or equivalently $\mathcal{L}_\xi \eta = 0$, and that $d\eta$ is $\varphi$-anti-invariant, namely 
\[
d\eta(\varphi X, \varphi Y) = -\, d\eta(X, Y) \qquad \forall\, X,Y \in \mathfrak{g}.
\]
Furthermore, an aqS structure satisfies
\[
 \mathcal{L}_\xi d\eta = 0, \qquad\mathcal{L}_\xi \varphi = 0, \qquad \mathcal{L}_\xi g = 0.
\]

A \emph{double aqS–Sasakian structure} on a Lie algebra $\mathfrak{g}$ 
is defined by three almost contact metric structures $(\varphi_i, \xi, \eta, g)$, 
$i = 1, 2, 3$, satisfying
\[
\varphi_i \varphi_j = \varphi_k = -\, \varphi_j \varphi_i,
\]
for every even permutation $(i, j, k)$ of $(1, 2, 3)$, and such that
\[
d\Phi_1 = 0, \qquad d\Phi_2 = 0, \qquad d\eta = 2\Phi_3.
\]
For a double aqS–Sasakian Lie algebra, one has $\dim \mathfrak{g} = 4n + 1$. 
One can also show that in this case, the first two structures $(\varphi_1, \xi, \eta, g)$ and $(\varphi_2, \xi, \eta, g)$ 
are anti-quasi-Sasakian, while $(\varphi_3, \xi, \eta, g)$ is Sasakian. The Heisenberg Lie algebra of dimension $4n+1$ can be endowed with a double aqS-Sasakian structure. We recall it in the $5$-dimensional case:

\begin{example}
 Let $\mathfrak{g}$ be the $5$-dimensional Lie algebra with basis $\lbrace e_1, e_2, e_3, e_4,\xi \rbrace$ and non-vanishing commutators 
 $${[e_1,e_4]=[e_2,e_3]=2 \xi}.$$ 
 Consider three 
        almost contact metric structures ${(\varphi_i,\xi,\eta,g)}$, $i=1,2,3$, where $g$ is the inner product with respect to which the basis is orthonormal, $\eta$ is the dual form of $\xi$, and $\varphi_ i$ is defined by 
        $${\varphi_i =e^1\otimes e_{i+1}-e^{i+1}\otimes e_{1}+e^{j+1}\otimes e_{k+1}-e^{k+1}\otimes e_{j+1}},$$
        where $e^l$ $(l=1,..,4)$ is the dual $1$-form of $e_l$, and $(i,j,k)$ is even permutation of $(1,2,3)$. Explicitly, ${\varphi_i(\xi)=0}$ and 
    \begin{eqnarray*} 
\begin{aligned}
   \varphi_1e_1&=e_2&\varphi_1e_2&=-e_1&\varphi_1e_3&=e_4&\varphi_1e_4&=-e_3\\
   \varphi_2e_1&=e_3&\varphi_2e_2&=-e_4&\varphi_2e_3&=-e_1&\varphi_2e_4&=e_2\\
  \varphi_3e_1&=e_4&\varphi_3e_2&=e_3&\varphi_3e_3&=-e_2&\varphi_3e_4&=-e_1.
\end{aligned}
\end{eqnarray*}
Clearly, $${\varphi_i\varphi_j=\varphi_k=-\varphi_j\varphi_i}.$$
The fundamental $2$-forms are $${\Phi_i=-(e^1\wedge e^{i+1}+e^{j+1}\wedge e^{k+1}).}$$
The $1$-forms $e^l$ are closed for every $l=1,..,4$, and then ${d\Phi_i=0}$ for every $i=1,2,3$. Furthermore, 
$d\eta=-2(e^1\wedge e^4+e^2\wedge e^3)=2\Phi_3.$
Thus, $(\mathfrak{g},\varphi_i,\xi,\eta,g)$, $i=1,2,3$, is a {double aqS-Sasakian} Lie algebra. In particular, the Sasakian Lie algebra $(\mathfrak{g},\varphi_3,\xi,\eta,g)$ is isomorphic to the Lie algebra
\[\mathfrak{g}_1=(0,0,0,0,e^{12}+e^{34})\simeq {\mathfrak{h}_5}\]
of Theorem \ref{theo-Sasaki-nontrivial}, where the Sasakian structure $(\varphi,\xi,\eta,g)$ is defined by 
\[\xi=e_5,\quad \eta=e^5,\quad \varphi\xi=0,\quad \varphi e_1=-e_2,\quad\varphi e_2=e_1,\quad\varphi e_3=-e_4,\quad\varphi e_4=e_3,\quad\]
\[ 2g(\cdot,\cdot)= d e^5(\varphi\cdot,\cdot)+2e^5\otimes e^5(\cdot,\cdot).\]
\end{example}

As remarked in \cite{DG}, in dimension $5$, double aqS-Sasakian structures are in one-to-one correspondence with contact Calabi-Yau structures as defined in \cite{TV}. A $5$-dimensional Lie algebra $\mathfrak{g}$ is \emph{contact Calabi-Yau} if it is endowed with a Sasakian structure $(\varphi,\xi,\eta,g)$ and a non-vanishing basic complex $(2,0)$-form $\epsilon$ on $\mathfrak{h}=\operatorname{ker}\eta$ such that  \[ d\epsilon = 0,\qquad
\epsilon \wedge \bar{\epsilon} = \frac{1}{2}(d\eta)^2.
\]
Here the form $\epsilon$ is basic if $\iota_\xi\epsilon=0$ and $\mathcal{L}_\xi\epsilon=0$. Now, if $(\mathfrak{g},\varphi_i,\xi,\eta,g)$ is a double aqS-Sasakian structure, one has a Sasakian structure $(\varphi_3,\xi,\eta,g)$ and putting $\omega_i=-\Phi_i$, $i=1,2$, the complex form $\epsilon=\omega_1+i\omega_2$ provides a contact Calabi-Yau structure. Conversely, given a contact Calabi-Yau Lie algebra $(\mathfrak{g},\varphi,\xi,\eta,g,\epsilon)$, a double aqS-Sasakian  structure is obtained by $\omega_1=\Re(\epsilon)$, $\omega_2=\Im(\epsilon)$ and $\omega_3=-\frac12d\eta$.  Before our main result, we show the following. 
\begin{lemma}\label{lemma-calabi-yau}
Any $5$-dimensional contact Calabi-Yau Lie algebra is isomorphic to the Heisenberg Lie algebra $\mathfrak{h}_5$.
\end{lemma}
\begin{proof}
The result is proved in \cite[Proposition 4.5]{AF2} under the assumption that the contact Calabi-Yau Lie algebra $\mathfrak{g}$ satisfies  $[\mathfrak{g},\mathfrak{g}]\neq \mathfrak{g}$.
This assumption can be actually removed, owing to the fact that a contact Calabi-Yau  Lie algebra admits an underlying Sasakian structure. Indeed, according to \cite{Diatta}, the only contact Lie algebra $\mathfrak{g}$ with trivial center such that $[\mathfrak{g},\mathfrak{g}]=\mathfrak{g}$ is the semidirect product $\mathfrak{sl}(2,\mathbb{R})\ltimes \mathbb{R}^2$. By [\citenum{AF}, Proposition 12], this Lie algebra does not admit any Sasakian structure (see also Theorem \ref{theo-Sasaki-trivial}). On the other hand, it follows from Theorem \ref{theo-Sasaki-nontrivial} that there exists no $5$-dimensional Sasakian Lie algebra with nontrivial center satisfying $[\mathfrak{g},\mathfrak{g}]=\mathfrak{g}$.
\end{proof}

Building on the above, we can classify $5$-dimensional aqS Lie algebras.
\begin{theorem} \label{theorem3}
    Let   $(\mathfrak{g},\varphi,\eta,\xi,g)$ be a $5$-dimensional anti-quasi-Sasakian (non cok\"ahler) Lie algebra. Then, up to a homothetic deformation of the structure, $\mathfrak{g}$ is isomorphic to the {Heisenberg} Lie algebra $\mathfrak{h}_5$. 
    \end{theorem}   
\begin{proof}
We apply general properties of anti-quasi-Sasakian structures (see \cite{DG} for details). For an aqS Lie algebra, one can consider two endomorphisms $\psi$ and $A$ given by  
\[
2g(X,\psi Y) = d\eta(X,Y), \quad A(X) = \varphi\psi(X), \quad \forall X,Y \in \mathfrak{g}.
\]
They are both skew-symmetric and anticommute with $\varphi$. In dimension 5, if the structure is non cok\"ahler, one has
$\psi^2 = A^2 = \lambda^2(-I + \eta \otimes \xi)$, $ \lambda \neq 0.$
The associated $2$-forms,
\[
\mathcal{A}(X,Y) = g(X, AY), \qquad \Psi(X,Y) = g(X, \psi Y),
\]
together with the fundamental $2$-form $\Phi$, satisfy the following relations:
\[
d\mathcal{A} = 0, \qquad d\Phi = 0, \qquad d\eta = 2\Psi.
\]
Next, consider the homothetic deformation
\[
\varphi' = \varphi, \qquad \xi' = \tfrac{1}{\lambda}\xi, \qquad \eta' = \lambda \eta, \qquad g' = \lambda^2 g.
\]
This is an anti-quasi-Sasakian structure with associated operators
\[
A' = \tfrac{1}{\lambda}A, \qquad \psi' = \tfrac{1}{\lambda}\psi.
\]
It follows that $\psi'^2 = A'^2 = -I + \eta' \otimes \xi',
$ and therefore $(A', \varphi', \psi', \xi', \eta', g')$ defines a double  aqS-Sasakian structure. This in turn determines a contact Calabi-Yau structure as remarked above. By Lemma \ref{lemma-calabi-yau}, it follows that the Lie algebra $\mathfrak{g}$ is isomorphic to the Heisenberg Lie algebra $\mathfrak{h}_5$. 
\end{proof}

 \section{Transversely K\"ahler almost contact metric structures}\label{Sect}
In this section, we study transversely K\"ahler almost contact metric Lie algebras, analogously to the same notion for manifolds. This class in particular includes both quasi Sasakian and anti-quasi-Sasakian Lie algebras. We examine transversely K\"ahler almost contact metric Lie algebras $(\mathfrak{g},\varphi,\xi,\eta,g)$ such that $\eta$ is a contact form, distinguishing between cases where the center is trivial or nontrivial. In the latter case, we investigate some curvature properties.

\subsection{Definition and basic properties}

A \emph{transversely
K\"ahler} almost contact metric manifold is a manifold $(M,\varphi,\xi,\eta,g)$ for which the structure tensor fields $(\varphi,g)$ are projectable along the $1$-dimensional foliation
generated by $\xi$ and induce a transverse K\"ahler structure. Referring to the Chinea-Gonzalez classification, this is the class of manifolds belonging to $\mathcal{C}_6\oplus\mathcal{C}_7\oplus\mathcal{C}_{10}\oplus\mathcal{C}_{11}$, for which the Lie derivative $\mathcal{L}_\xi\varphi$ vanishes (\cite{Chinea-Gonzalez,DG,ADDK}). In analogy with the characterization of transversely K\"ahler almost contact metric manifolds given in \cite{DG}, we introduce the notion in the context of Lie algebras.

\begin{definition}\label{def1}
   An almost contact metric structure $(\varphi,\xi,\eta,g)$ on a Lie algebra $\mathfrak{g}$ is \emph{transversely K\"ahler} if 
\[
d\Phi=0, \quad N_\varphi(\xi,X)=0, \quad N_\varphi(X,Y,Z)=0, \quad \forall X,Y,Z\in\mathfrak{h},
\]
  where $N_\varphi(X,Y,Z):=g(N_\varphi(X,Y),Z).$ In particular both quasi Sasakian and anti-quasi Sasakian structures are {transversely K\"ahler}. If furthermore, the $1$-form $\eta$ is a contact form, i. e. ${\eta \wedge(d\eta)^n\neq 0}$, we say that the structure is of \emph{maximal rank}.
\end{definition}

\begin{remark}
The conditions in Definition \ref{def1} can equivalently be expressed as
\[
(\mathcal{L}_\xi \varphi)(X)=0, \quad (\mathcal{L}_\xi g)(X,Y)=0, \quad N_\varphi(X,Y,Z)=0, \quad d\Phi(X,Y,Z)=0, 
\]
for all $X,Y,Z\in\mathfrak{h}$. The equivalence is shown in \cite{DG} for transversely K\"ahler structures on manifolds, and essentially depends on the identities
\begin{equation}\label{N(X,xi)}
    N_\varphi(\xi,X)=-\varphi(\mathcal{L}_\xi\varphi)X+d\eta(\xi,X)\xi,
\end{equation}
\begin{equation}\label{dPhi}
    d\Phi(\xi,X,Y)=(\mathcal{L}_\xi g)(X,\varphi Y)+g(X,(\mathcal{L}_\xi\varphi)Y).
\end{equation}
It is also worth remarking that the intersection between transversely K\"ahler  and $K$-contact structures is exactly given by Sasakian structures.
\end{remark}
\begin{remark} For a transversely K\"ahler almost contact metric structure, one has \[\mathcal{L}_\xi\eta=0, \qquad  \mathcal{L}_\xi d\eta=0, \qquad \mathcal{L}_\xi\varphi=0 , \qquad  \mathcal{L}_\xi g=0\]
which can be rephrased as follows: 
\begin{eqnarray*}
\mathcal{L}_\xi\eta=0 &\Leftrightarrow& \eta\circ \ad_\xi=0,\\
\mathcal{L}_\xi d\eta=0 &\Leftrightarrow& d\eta(\ad_\xi X,Y)+d\eta(X,\ad_\xi Y)=0\quad \forall X,Y\in\mathfrak{g},\\
\mathcal{L}_\xi \varphi = 0  &\Leftrightarrow &\ad_\xi\circ\, \varphi = \varphi \circ  \ad_\xi\quad  \text{($\ad_\xi$ commutes with $\varphi$)},\\
\mathcal{L}_\xi g = 0 &\Leftrightarrow &g(\ad_\xi X,Y)+g(X,\ad_\xi Y)=0\quad \forall X,Y\in\mathfrak{g}\\ &&\text{($\ad_\xi$ is skew-symmetric with respect to $g$)}
\end{eqnarray*}
\end{remark}

For  transversely K\"ahler almost contact metric Lie algebras, the following structural result holds.

 \begin{proposition}\label{prop6}
      Let $(\mathfrak{g},\varphi,\eta,\xi,g)$ be a transversely K\"ahler a.c.m. Lie algebra. Then,  $\ker\ad_\xi$ and $\im\ad_\xi$ are $\varphi$-invariant subspaces of $\mathfrak{g}$ satisfying $(\im\ad_\xi)^\perp=\ker\ad_\xi$. Therefore,
 \begin{equation}\label{decomposition}
 \mathfrak{g}=\ker\ad_\xi \oplus \im\ad_\xi.
 \end{equation}
 \end{proposition}
\begin{proof}
Since $\ad_\xi \circ\,\varphi = \varphi \circ \ad_\xi$, both $\ker \ad_\xi$ and $\im\ad_\xi$ are $\varphi$-invariant. 
Moreover, since $\mathcal{L}_\xi g = 0$, for any $X \in \ker \ad_\xi$ and $Y \in \mathfrak{g}$, we have
$0  = -g(X, [\xi,Y]),
$
which shows that $X \in (\im \ad_\xi)^\perp$. Hence,
$\ker \ad_\xi \subseteq (\im \ad_\xi)^\perp$. 
Comparing dimensions shows that $\ker \ad_\xi = (\im \ad_\xi)^\perp$, establishing the decomposition 
$\mathfrak{g} = \ker \ad_\xi \oplus \im\ad_\xi$.
\end{proof}

\begin{remark}
The fact that Equation \eqref{decomposition} holds exactly as in the Sasakian case, makes it possible to treat the $3$-dimensional case in the same manner as $3$-dimensional Sasakian Lie algebras are classified in [\citenum{AF}, Theorem 9]. Arguing analougously, one can verify that 
    any $3$-dimensional transversely K\"ahler almost contact metric Lie algebra, with contact form $\eta$, is isomorphic to one of the Lie following: $\mathfrak{su}(2)$, $\mathfrak{sl}(2,\mathbb{R})$, $\mathfrak{aff}(\mathbb{R})\times\mathbb{R}$ and  $\mathfrak{h}_3$. 
\end{remark}

It has been shown in [\citenum{AF}, Proposition 1] that the center of a contact Lie algebra $(\mathfrak{g},\eta)$ is either trivial or one-dimensional; in the latter case, its center 
$\mathfrak{z}(\mathfrak{g})$  is spanned by the Reeb vector $\xi$. Accordingly, any transversely K\"ahler a.c.m. Lie algebra of maximal rank with a nontrivial center satisfies
$\mathfrak{z}(\mathfrak{g}) = \mathbb{R}\xi.$

\subsection{The case of trivial center}
Now we investigate transversely K\"ahler a.c.m. Lie algebras of maximal rank with trivial center.

\begin{proposition}\label{prop8}
     Let $(\mathfrak{g},\varphi,\eta,\xi,g)$ be a transversely K\"ahler a.c.m. Lie algebra of maximal rank with trivial center.
        
        \begin{itemize}
            \item[(1)] If $\dim\mathfrak{g}\geq 5$, then $\ker\ad_\xi$ is a transversely K\"ahler  a.c.m. Lie subalgebra of $\mathfrak{g}$ with maximal rank and nontrivial center.
            \item[(2)] If $X\in \ker\ad_\xi$, $Y\in \im\ad_\xi$, then $[X,Y]\in \im\ad_\xi$.
        \end{itemize}
\end{proposition}

\begin{proof}
For the first item, using the Jacobi identity, we have $[[X,Y],\xi]=0$ for any $X,Y \in \ker \ad_\xi$, which shows that $[X,Y] \in \ker \ad_\xi$ and hence that $\ker \ad_\xi$ is a Lie subalgebra of $\mathfrak{g}$.

  Let $X \in \ker \ad_\xi \cap \ker \eta$ and $Y \in \mathfrak{g}$. Then, being $\mathcal{L}_\xi d\eta=0$,  $d\eta(X, \ad_\xi Y)=-d\eta(\ad_\xi X,Y)=0$, which implies that $d\eta(X, Z')=0$ for any $Z' \in \im \ad_\xi$. Now, if $d\eta(X, Z)=0$ for all $Z \in \ker \mathrm{ad}_\xi$, then $$d\eta(X, W)=0$$ for any $W \in \ker \eta$, and consequently $X=0$, since
$d\eta$ is nondegenerate on $\ker \eta$. Therefore, the restriction of $\eta$ to $\ker \ad_\xi$ is a contact form on the Lie subalgebra $\ker \ad_\xi$. Moreover, $\xi \in \mathfrak{z}(\ker \ad_\xi)$, and the restrictions of $\varphi$ and $g$ induce a transversely K\"ahler almost contact metric structure on $\ker\ad_\xi$.

         For the second item, one can write $$Y=[\xi,Y'],$$
        with $Y'\in \mathfrak{g}$. From the Jacobi identity,
        $$[X,Y]=[X,[\xi,Y']]=-[\xi,[Y',X]]-[Y',[X,\xi]]=-[\xi,[Y',X]]$$
        which is contained in $\im\ad_\xi$.
\end{proof}

     Observe that if $\ker\eta = \im\ad_\xi$, then the commutator ideal $\mathfrak{g}'=[\mathfrak{g},\mathfrak{g}]$ coincides with $\mathfrak{g}$. Assume now that the transversely K\"ahler a.c.m. Lie algebra $(\mathfrak{g},\varphi,\eta,\xi,g)$ of maximal rank with trivial center further satisfies $\mathfrak{g}'\neq \mathfrak{g}$. Under the decomposition $\mathfrak{g} = \ker \ad_\xi \oplus \im \ad_\xi$, the restrictions of $\ad_\xi$ and $\varphi$ to $\ker \eta$ can be expressed as
\[
(\ad_\xi)|_{\ker \eta} = 
\begin{pmatrix}
0 & 0 \\
0 & U
\end{pmatrix}, 
\hspace{1cm}
\varphi|_{\ker \eta} = 
\begin{pmatrix}
A & C \\
B & D
\end{pmatrix},
\]
where $U : \im\ad_\xi \to \im\ad_\xi$ is a nonsingular operator.
     From the relation $\varphi \circ\ad_\xi = \ad_\xi\circ \varphi$, it follows that
\[
B = C = 0, \qquad DU = UD,
\]
and, using $(\varphi|_{\ker \eta})^2 = -I$, we also obtain
\[
A^2 = D^2 = -I.
\]

Notice that, for a solvable $\mathfrak{g}$, the Reeb vector $\xi$ cannot belong to the commutator  $\mathfrak{g}'$. Indeed, if one assumes $\xi\in\mathfrak{g}'$, by the Cartan criterion of solvability, one has $\operatorname{tr}(\ad_\xi^2)=0$, and the skew-symmetry of $\ad_\xi$ implies that $\xi$ is central, which is a contradiction.

 \subsection{The case of nontrivial center}

 In this subsection, we investigate transversely K\"ahler almost contact metric Lie algebras of maximal rank and with nontrivial center, establishing a one-to-one correspondence with  K\"ahler Lie algebras endowed with a symplectic form. The following theorem is a direct consequence of results in \cite{Dario}. For the sake of completeness, we explicitly detail certain constructions as required for our purposes.

\begin{theorem} \label{theorem6}
    There exists a one-to-one correspondence between $(2n+1)$-dimensional transversely K\"ahler a.c.m. Lie algebras $(\mathfrak{g},\varphi,\xi,\eta,g)$ of maximal rank with nontrivial center, and $2n$-dimensional K\"ahler Lie algebras $(\mathfrak{h},J,h,\omega)$ endowed with a symplectic $2$-form.
\end{theorem}
\begin{proof}
    We begin by assuming that $(\mathfrak{g}, \varphi, \xi, \eta, g)$ is a transversely K\"ahler a.c.m. Lie algebra of maximal rank and nontrivial center. Since $\eta(\xi)=1$ and $d\eta(\xi,\cdot)=0$, $\xi$ coincides with the Reeb vector associated with the contact form $\eta$. As the center of $\mathfrak{g}$ is nontrivial, we have $\mathfrak{z}(\mathfrak{g})=\mathbb{R}\xi$.

Furthermore, as vector spaces $\mathfrak{g}=\mathfrak{h}\oplus\mathbb{R}\xi=\mathfrak{h}\oplus \mathfrak{z}(\mathfrak{g})$ and $\mathfrak{g} / \mathfrak{z}(\mathfrak{g}) \cong \mathfrak{h}$ by means $X+ \mathfrak{z}(\mathfrak{g}) \mapsto 
 X_\mathfrak{h}$. Being  $ \mathfrak{z}(\mathfrak{g})$ an ideal of $\mathfrak{g}$, $\mathfrak{g} / \mathfrak{z}(\mathfrak{g})$ is a Lie algebra and the above isomorphism induces the Lie bracket $[\cdot, \cdot]_\mathfrak{h}$ on $\mathfrak{h}$. Additionally, one has 
 $$[X,Y]=[X,Y]_\mathfrak{h}+\eta([X,Y])\xi=[X,Y]_\mathfrak{h}-d\eta(X,Y)\xi,$$
 for every $X,Y \in \mathfrak{h}$, so that $\mathfrak{g}$ is the $1$-dimensional central extension of $(\mathfrak{h},[\cdot,\cdot]_\mathfrak{h})$ by the $2$-cocycle $-d\eta$.

If we take $J = \varphi|_\mathfrak{h}$ and $h = g|_{\mathfrak{h} \times \mathfrak{h}}$, then $(J,h)$ is an almost Hermitian structure on the Lie algebra $(\mathfrak{h}, [\cdot, \cdot]_\mathfrak{h})$. For all $X, Y \in \mathfrak{h}$, we have
\begin{eqnarray*}
    N_J(X,Y)
    &=& [\varphi X, \varphi Y]_\mathfrak{h} + \varphi^2[X,Y]_\mathfrak{h} - \varphi[X, \varphi Y]_\mathfrak{h} - \varphi[\varphi X, Y]_\mathfrak{h} \\
    &=& (N_\varphi(X,Y))_\mathfrak{h} = 0,
\end{eqnarray*}
which shows that $J$ is a complex structure. Denoting by $\Omega = \Phi|_{\mathfrak{h} \times \mathfrak{h}}$ the fundamental $2$-form associated with $(J,h)$, one has 
\[ d\Omega(X,Y,Z)=d\Phi(X,Y,Z)=0\]
for all $X,Y,Z \in \mathfrak{h}$ and hence $(\mathfrak{h},J,h)$ is K\"ahlerian.

Conversely, starting with a Kähler Lie algebra $(\mathfrak{h}, [\cdot, \cdot]_\mathfrak{h}, J, h)$ and assuming 
      $\omega$ is a symplectic 2-form on $\mathfrak{h}$, one can consider the $1$-dimensional central extension $\mathfrak{g} := \mathfrak{h} \oplus \mathbb{R}\xi$ of $\mathfrak{h}$, whose Lie bracket $[\cdot,\cdot]$ is defined by
\begin{eqnarray} 
[X,\xi]=0, \quad [X,Y]=[X,Y]_{\mathfrak{h}}-\omega(X,Y)\xi,\label{1}
\end{eqnarray}
for all  $X,Y \in \mathfrak{h}$.
Thus, $\mathfrak{g}$ is equipped with an almost contact metric structure $(\varphi, \xi, \eta, g)$, 
where $\eta$ is the $1$-form dual to $\xi$, and $(\varphi, g)$ are obtained by extending the K\"ahler structure $(J, h)$ 
so that $\mathfrak{h} = \ker \eta = \im\varphi$, and $\xi$ is unit and orthogonal to $\mathfrak{h}$. Precisely, for all $X,Y,Z \in \mathfrak{h}$
$$\varphi (\xi)=0, \quad \varphi(X)=JX$$
and 
$$g(\xi,\xi)=1, \quad g(\xi,X)=0, \quad g(X,Y)=h(X,Y).$$
The structure has maximal rank, being $d\eta|_{\mathfrak{h}\times\mathfrak{h}}=\omega$.

We show that the structure is transversely K\"ahler. Indeed, the fundamental $2$-form $\Phi=g(\cdot, \varphi \cdot)$ is closed. This is a consequence of the closure of $\Omega$ on $\mathfrak{h}$ and the fact that $\xi$ lies in the center of $\mathfrak{g}$. As regards, the tensor $N_\varphi$, using \eqref{N(X,xi)} one has 
$$N_\varphi(\xi,X)=0,$$ 
and using the fact that $J$ is a complex structure, for every $X,Y\in\mathfrak{h}$
\begin{equation}
    N_\varphi(X,Y)=-\omega(JX,JY)\xi+\omega(X,Y)\xi.\label{eqn4.4}
\end{equation}
    Consequently, for every $X,Y,Z \in \mathfrak{h}$, $N_\varphi(X,Y,Z)=0$, which completes the proof that  $(\varphi,\eta,\xi,g)$ is a transversely K\"ahler almost contact metric structure. 
\end{proof}

  \begin{remark}\label{remark aqS qS}

Considering the structure $(\varphi,\xi,\eta,g)$ in the above theorem, from equation (\ref{eqn4.4}), we have
\begin{itemize}
    \item the structure is anti-quasi-Sasakian if and only if $N_\varphi=2d\eta\otimes\xi$, namely $\omega(JX,JX)=-\omega(X,Y)$;
    \item the structure is quasi Sasakian if and only if $N_\varphi=0$, namely $\omega(JX,JX)=\omega(X,Y)$. It is $\alpha$-Sasakian iff $\omega=2\alpha\Omega$.
\end{itemize}
 \end{remark}
\begin{remark}\label{remark aqS}
 We showed in Theorem \ref{theorem3} that every $5$-dimensional anti-quasi-Sasakian (non cok\"ahler) Lie algebra  $(\mathfrak{g},\varphi,\xi, \eta,g)$ is isomorphic to the {Heisenberg} Lie algebra $\mathfrak{h}_5$. In the case where the Lie algebra has nontrivial center, the result can be alternatively justified using Theorem \ref{theorem6}. Indeed, considering a $5$-dimensional anti-quasi-Sasakian non cok\"ahler Lie algebra  $(\mathfrak{g},\varphi,\xi,\eta,g)$, it is already known that $\eta$ is contact form, and by Theorem \ref{theorem6}, the Lie algebra is the $1$-dimensional central extension of a $4$-dimensional K\"ahler Lie algebra $(\mathfrak{h},J,h,\omega)$ endowed with a symplectic $2$-form satisfying $\omega(JX,JY)=-\omega(X,Y)$. Then $(J,\omega)$ is what is called a complex symplectic structure on the Lie algebra $\mathfrak{h}$. The $4$-dimensional complex symplectic Lie algebras have been classified in \cite[Theorem 3.1]{BGL}. Combining this with the classification of $4$-dimensional K\"ahler Lie algebras obtained by Ovando in \cite{Ovandop} (see also Section \ref{section-dim5-center} and Table \ref{table1}) one can deduce that the only possibility is that $\mathfrak{g}$ is the central extension of the abelian K\"ahler Lie algebra $\mathbb{R}^4$. Therefore, for every horizontal vectors $X,Y$,
 \[[X,Y]=-\omega(X,Y)\xi,\]
 where $\omega$ is a complex symplectic structure on $\mathbb{R}^4$. One can easily show (see \cite[Lemma 1.6]{DG}) that $\mathbb{R}^4$ admits a basis $(e,f,Je,Jf)$ such that 
 \[\omega(e,f)=-\omega(Je,Jf)=a\ne0,\qquad \omega(e,Jf)=\omega(Je,f)=0,\]
 so that the nonvanishing brackets in the Lie algebra $\mathfrak{g}$ are 
 \[[f,e]=[Je,Jf]=a\xi\]
 and $\mathfrak{g}$ is isomorphic to $\mathfrak{h}_5$.
\end{remark}
From the proof of the Theorem \ref{theorem6} we also have the following result:
\begin{corollary}
    If $(\mathfrak{g},\varphi,\eta,\xi,g)$ is a transversely K\"ahler almost contact metric Lie algebra of maximal rank and nontrivial center $\mathfrak{z(g)}$, then $\mathfrak{g}/\mathfrak{z(g)}$ inherits a natural K\"ahler structure.
\end{corollary}
Furthermore, the one-to-one correspondence of transversely K\"ahler almost contact metric Lie algebras of maximal rank with nontrivial center and K\"ahler Lie algebras leads to the following Proposition.
\begin{proposition}\label{proposition-isom}
    Let $(\mathfrak{g}_1,\varphi_1,\xi_1,\eta_1,g_1)$ and $(\mathfrak{g}_2,\varphi_2,\xi_2,\eta_2,g_2)$ be two transversely K\"ahler almost contact metric Lie algebras obtained as the central extensions of K\"ahler Lie algebras $(\mathfrak{h}_1,J_1,h_1)$ and $(\mathfrak{h}_2,J_2,h_2)$ via symplectic $2$-forms $\omega_1$ and $\omega_2$, respectively. Then $(\mathfrak{g}_1,\varphi_1,\xi_1,\eta_1,g_1)$ and $(\mathfrak{g}_2,\varphi_2,\xi_2,\eta_2,g_2)$ are isomorphic if and only if there exists an isomorphism $f:\mathfrak{h}_1\to\mathfrak{h}_2$ of K\"ahler Lie algebras such that $\omega_2(fX,fY)=\omega_1(X,Y)$ for every $X,Y\in\mathfrak{h}_1$, namely $f$ is a symplectomorphism.
\end{proposition}
\begin{proof}
    If $f:\mathfrak{g}_1\to\mathfrak{g}_2$ is an isomorphism of transversely K\"ahler almost contact metric Lie algebras, then $f\mathfrak(\mathfrak{h}_1)=\mathfrak{h}_2$ and, for every $X,Y\in\mathfrak{h}_1$, one has
    \[[fX,fY]_{\mathfrak{h}_2}-\omega_2(fX,fY)\xi_2=[fX,fY]=f[X,Y]=f([X,Y]_{\mathfrak{h}_1})-\omega_1(X,Y)f(\xi_1).\]
    Since $f(\xi_1)=\xi_2$, comparing the horizontal and vertical components, one has $[fX,fY]_{\mathfrak{h}_2}=f([X,Y]_{\mathfrak{h}_1})$ and $\omega_2(fX,fY)=\omega_1(X,Y)$. In particular $f:\mathfrak{h}_1\to\mathfrak{h}_2$ is an isomorphism of K\"ahler Lie algebras. Conversely, every isomorphism of K\"ahler Lie algebras, which is also a symplectomorphism, extends to an isomorphism of almost contact metric Lie algebras defined setting $f(\xi_1)=\xi_2.$
\end{proof}

\subsection{ Ricci tensors in the case of nontrivial center}\label{section-curvature}

\indent Now, we investigate the curvature relations between a K\"ahler Lie algebra ${(\mathfrak{h},J,h,\omega)}$ endowed with a symplectic structure $\omega$, and its  central extension $\mathfrak{g}:=\mathfrak{h}\oplus \mathbb{R}\xi$ with the structure ${(\varphi, \xi,\eta,g)}$.\\

If {$G$} and {$H$} are Lie groups with Lie algebras $\mathfrak{g}$ and $\mathfrak{h}$ respectively, then $G$ and $H$ are endowed with left-invariant structures {$(\varphi,\xi,\eta,g)$} and {$(J,h,\omega)$} satisfying analogous properties. Denote by $\nabla$ and $\nabla^{\mathfrak{h}}$ the Levi-Civita connections on $G$ and $H$, respectively.\\

We use the Koszul formula given, for left invariant vector fields $X,Y,Z$, by
\[ 2g(\nabla_X Y,Z)=g([X,Y],Z)-g([Y,Z],X)+g([Z,X],Y).\]
Using the fact that $\xi$ is central, this implies 
\begin{eqnarray}
\nabla_{\xi}\xi=0, ~~~ \nabla_X \xi=\nabla_{\xi}X \in \mathfrak{h}\quad\forall X\in\mathfrak{h}.\label{1.3}
\end{eqnarray}
Furthermore, for every $X,Y,Z\in\mathfrak{h}$
\[g(\nabla_X Y,Z)=g(\nabla^\mathfrak{h}_X Y,Z)\]
and using (\ref{1}), one can write
\[2g(\nabla_XY,\xi)=g([X,Y],\xi)=-\omega(X,Y).\]
Then, for all $X,Y \in \mathfrak{h}$ we get
\begin{eqnarray}
\nabla_{X}Y=\nabla_X^{\mathfrak{h}}Y -\frac{1}{2}\omega(X,Y)\xi. \label{1.5}
\end{eqnarray}

\begin{lemma} \label{lem9}
     Let $(\mathfrak{g},\varphi,\xi,\eta,g)$ be a transversely K\"ahler almost contact metric Lie algebra of maximal rank as a central extension of a K\"ahler Lie algebra $(\mathfrak{h},J,h,\omega)$ with a symplectic structure $\omega$. Let $R$ and $R^{\mathfrak{h}}$ be the curvature tensors of $\nabla$ and $\nabla^\mathfrak{h}$, respectively. Then:
     \begin{eqnarray}\nonumber
 R(X,Y,Z,W)&=&R^{\mathfrak{h}}(X,Y,Z,W)+\frac{1}{4}\Big(2\omega(X,Y)\omega(Z,W)\\ \label{cg1.1}
    &&-\omega(Y,Z)\omega(X,W)+\omega(X,Z)\omega(Y,W) \Big),\\ \label{cg1.2}
 R(X,Y,Z,\xi)&=&-\frac{1}{2}\Big(\omega(X,\nabla_Y^{\mathfrak{h}}Z)-\omega(Y,\nabla_X^{\mathfrak{h}}Z)-\omega([X,Y]_{\mathfrak{h}},Z)\Big),\\ \label{cg1.3}
 R(X,\xi,\xi,Z)&=& \frac{1}{4}\sum_i\omega(X,X_i)\omega(Z,X_i),  
\end{eqnarray}
for all $X,Y,Z,W \in \mathfrak{h}$, where  and $\{X_i\}_{i=1,\ldots,2n}$ denotes an orthonormal basis of the Lie algebra $\mathfrak{h}$.
 \end{lemma}  
\begin{proof}
Using \eqref{1.5} and \eqref{1.3}, for every $X,Y,Z\in\mathfrak{h}$, we compute
\begin{eqnarray*}\nonumber
R(X,Y)Z&=&\nabla_X\left(\nabla^\mathfrak{h}_YZ-\frac12\omega(Y,Z)\xi\right)-\nabla_Y\left(\nabla^\mathfrak{h}_XZ-\frac12\omega(X,Z)\xi\right)\\
&&-\nabla_{[X,Y]_\mathfrak{h}-\omega(X,Y)\xi}Z\\
&=& R^{\mathfrak{h}}(X,Y)Z
-\frac{1}{2}\Big(\omega(X,\nabla_Y^{\mathfrak{h}}Z)- \omega(Y,\nabla_X^{\mathfrak{h}}Z)-\omega({[X,Y]_{{\mathfrak{h}}}},Z)\Big)\xi\\
&&{}-\frac12\left(\omega(Y,Z)\nabla_X\xi-\omega(X,Z)\nabla_Y\xi-2\omega(X,Y)\nabla_Z\xi\right).
\end{eqnarray*}
Taking in the above equation the scalar product with $W\in\mathfrak{h}$, and using $g(\nabla_X\xi,W)=\frac12\omega(X,W)$, we get \eqref{cg1.1}. Taking the scalar product with $\xi$, we get \eqref{cg1.2}.
Now, for any $X\in\mathfrak{h}$, being $\nabla_\xi\xi=0$ and $[X,\xi]=0$, we have
\[R(X,\xi)\xi=-\nabla_\xi\nabla_X\xi.\]
We take the product with $Z\in\mathfrak{h}$. Using $\nabla_\xi Z=\nabla_ Z\xi $, and an orthonormal basis $\{X_i\}_{i=1,\ldots,2n}$ of $\mathfrak{h}$, we get
\begin{align*}R(X,\xi,\xi,Z)&=g( \nabla_X\xi,\nabla_Z\xi)=\sum_ig( \nabla_X\xi,X_i)g( \nabla_Z\xi,X_i)\\&=\frac14\sum_i\omega(X,X_i)\omega(Z,X_i),
\end{align*}
which proves \eqref{cg1.3}. 
\end{proof}

By using Lemma~\ref{lem9}, we obtain the following proposition.
 \begin{proposition} \label{proposition10}
        Let $(\mathfrak{g},\varphi,\xi,\eta,g)$ be a transversely K\"ahler almost contact metric Lie algebra of maximal rank as a central extension of a K\"ahler Lie algebra $(\mathfrak{h},J,h,\omega)$ with a symplectic structure $\omega$. The Ricci curvatures of the extended Lie algebra are given by:
\begin{align}
 \rho(Y,Z)&=\rho^{\mathfrak{h}}(Y,Z)+\frac{1}{2} \sum_i \omega(X_i,Y)\omega(Z,X_i), \label{cr.1} \\
 \rho(Y,\xi)&=\frac{1}{2} \sum_i \left(\omega(X_i,\nabla_{X_i}^{\mathfrak{h}}Y)+\omega(\nabla_{X_i}^{\mathfrak{h}}X_i,Y)\right), \label{cr.2} \\
\rho(\xi,\xi)&=\frac{1}{4} \sum_{i,j} \omega(X_i,X_j)^2, \label{cr.3}
\end{align}
for all $Y,Z \in \mathfrak{h}$, where $\{X_i\}_{i=1,\ldots,2n}$ denotes an orthonormal basis of $\mathfrak{h}$, and $\rho^\mathfrak{h}$ is the Ricci tensor of the K\"ahler Lie algebra.
 \end{proposition}

 \begin{proof}
Taking $X_i$ instead of $X$ and $W$ in (\ref{cg1.1}), one gets
\begin{eqnarray*}
R(X_i,Y,Z,X_i)&=&R^{\mathfrak{h}}(X_i,Y,Z,X_i)+\frac{1}{4}\Big(2\omega(X_i,Y)\omega(Z,X_i)\\
    &&{}-\omega(Y,Z)\omega(X_i,X_i)+\omega(X_i,Z)\omega(Y,X_i) \Big)\\
    &=& R^\mathfrak{h}(X_i,Y,Z,X_i)+\frac34\omega(X_i,Y)\omega(Z,X_i).
\end{eqnarray*}
Using also \eqref{cg1.3}, we compute
\begin{eqnarray*}    \rho(Y,Z)&=&\sum_iR(X_i,Y,Z,X_i)+R(\xi,Y,Z,\xi)\\
    &=&\rho^\mathfrak{h}(Y,Z)+\frac34\sum_i\omega(X_i,Y)\omega(Z,X_i)+\frac14\sum_i\omega(Y,X_i)\omega(Z,X_i)\\
    &=&\rho^\mathfrak{h}(Y,Z)+\frac12\sum_i\omega(X_i,Y)\omega(Z,X_i)
\end{eqnarray*}
which proves equation (\ref{cr.1}).
Analogously, replacing  $X$ and $Z$ with $X_i$ in (\ref{cg1.2}), we have
\begin{eqnarray*}    \rho(Y,\xi)&=&\sum_iR(X_i,Y,\xi,X_i)\\
&=&\frac12\sum_i\left(\omega(X_i,\nabla^\mathfrak{h}_YX_i)-\omega(Y,\nabla^\mathfrak{h}_{X_i}X_i)-\omega([X_i,Y]_\mathfrak{h},X_i)\right)\\
&=&\frac12\sum_i\left(\omega(X_i,\nabla^\mathfrak{h}_{X_i}Y-[X_i,Y]_\mathfrak{h})-\omega(Y,\nabla^\mathfrak{h}_{X_i}X_i)-\omega([X_i,Y]_\mathfrak{h},X_i)\right)\\
&=&\frac12\sum_i\left(\omega(X_i,\nabla^\mathfrak{h}_{X_i}Y)+\omega(\nabla^\mathfrak{h}_{X_i}X_i,Y)\right).
\end{eqnarray*}
Finally,  equation (\ref{cr.3}) follows immediately from \eqref{cg1.3}.
\end{proof}

\section{\ensuremath{5}-dimensional transversely K\"ahler almost contact metric Lie algebras of maximal rank}

In this section we investigate $5$–dimensional transversely K\"ahler almost contact metric Lie algebras of maximal rank, distinguishing again the cases whether the center is trivial or nontrivial. In each case we establish the conditions under which the Lie algebras admit an $\eta$-Einstein structure.

\subsection{The case of trivial center in dimension \ensuremath{5}}

\begin{theorem} \label{theorem13}
    If a $5$-dimensional transversely K\"ahler almost contact metric Lie algebra $\mathfrak{g}$ of maximal rank has trivial center, then it is isomorphic to one of the following Lie algebras:     $\mathfrak{aff}(\mathbb{R}) \times \mathfrak{sl}(2,\mathbb{R})$, $\mathfrak{aff}(\mathbb{R}) \times \mathfrak{su}(2)$, or a non-unimodular Lie algebra $\mathfrak{g}_0\simeq \mathbb{R}^2 \ltimes\mathfrak{h}_3$. Furthermore, the structure is always $\alpha$-Sasakian. 
\end{theorem}
\begin{proof}
    
 Let $(\mathfrak{g}, \varphi, \eta, \xi,g)$ be a $5$-dimensional transversely K\"ahler almost contact metric Lie algebra  of maximal rank and with trivial center.

Assume $\mathfrak{g}'=\mathfrak{g}$.
As stated in [\citenum{Diatta}, Sect.5] the only contact Lie algebra with trivial center and $\mathfrak{g}'=\mathfrak{g}$ is the 
semidirect product $\mathfrak{sl}(2,\mathbb{R})\ltimes\mathbb{R}^2,$ where $\mathfrak{sl}(2,\mathbb{R})$ acts on $\mathbb{R}^2$ by matrix multiplication. In \cite[Theorem 9]{AF} it is showed that this Lie algebra does not admit any Sasakian structure. Arguing exactly in the same way as in \cite{AF}, one obtains that there is no transversely K\"ahler a.c.m. structure of maximal rank on the Lie algebra $\mathfrak{sl}(2,\mathbb{R}) \ltimes \mathbb{R}^2$.

Assume now $\mathfrak{g}'\neq \mathfrak{g}$.
For this case,
$$\dim\ker(\ad_\xi)|_{\ker\eta}=\dim\im(\ad_\xi)=2.$$
An orthonormal basis $\{e_1, \ldots, e_4\}$ of $\ker \eta$ can easily be chosen so that $\varphi|_{\ker \eta}$ has the matrix form
\[
\varphi|_{\ker\eta} = \begin{pmatrix}
0 & 1 & 0 & 0 \\
-1 & 0 & 0 & 0 \\
0 & 0 & 0 & 1 \\
0 & 0 & -1 & 0
\end{pmatrix}
\]
with $\ker\ad_\xi=Span\lbrace\xi,e_1,e_2\rbrace$, $\im\ad_\xi=Span\lbrace e_3,e_4\rbrace$. In particular, the fundamental $2$-form is $\Phi=e^{12}+ e^{34}$. Continuing in this basis, and considering that $\ad_\xi$ is skew symmetric, we have
 \[
(\ad_\xi)|_{\ker\eta} = \begin{pmatrix}
0 & 0 & 0 & 0 \\
0 & 0 & 0 & 0 \\
0 & 0 & 0 & -b \\
0 & 0 &b & 0
\end{pmatrix}
\]  
with $b\ne 0$. Applying the Jacobi identity,
\begin{align*}
 0&=[\xi,[e_3,e_4]]+[e_4,[\xi,e_3]]-[e_3,[\xi,e_4]] \\
  &=[\xi,[e_3,e_4]]+[e_4,b e_4]-[e_3,-b e_3] =[\xi,[e_3,e_4]],
\end{align*}  
and thus $[e_3,e_4]\in\ker\ad_\xi$.
Using also Proposition \ref{prop8},   
one can write 
\begin{align}
[e_1,e_2] &= -a_1 e_1 - b_1 e_2 - k_1 \xi, \notag\\ [e_1,e_3] &= -c_2 e_3 - f_2 e_4, \notag \\
[e_1,e_4] &= -c_3 e_3 - f_3 e_4,\notag\\
[e_2,e_3] &= -c_4 e_3 - f_4 e_4, \label{eqnk} \\
[e_2,e_4] &= -c_5 e_3 - f_5 e_4, \notag\\ [e_3,e_4] &= -a_6 e_1 - b_6 e_2 - k_6 \xi, \notag \\
[e_3,\xi] &= - b e_4, \notag\\ [e_4,\xi] &= b e_3, \notag
\end{align}
  where $k_1\ne0$ and $k_6\ne0$, since $d\eta = k_1 e^{12} + k_6 e^{34}$ and $\eta$ is a contact form. We denote by $\{e^1, e^2, e^3, e^4, e^5=\eta\}$ the dual basis of $\{e_1, e_2, e_3, e_4, e_5=\xi\}$.

Since $b \neq 0$, we can assume $b = \pm 1$. This is possible up to a a homothetic deformation of type 
\begin{equation}\label{deformation}
    \varphi' = \varphi,
    \quad \xi' = \pm\tfrac{1}{b}\xi, \quad\eta' = \pm b  \eta,\quad  g' = b^2 g,
\end{equation}
 and obviously we work with a basis orthonormal with respect to the new metric. The subsequent analysis is divided into two distinct cases: \textbf{Case A:} $b = 1$ and \textbf{Case B:} $b = -1$.  \smallskip

We first focus on \textbf{Case A}. Using \eqref{eqnk}, we get
\begin{align}
 &de^1=a_1e^{12}+a_6e^{34},\nonumber \\
  &de^2=b_1e^{12}+b_6e^{34},\nonumber \\
 &de^3=c_2e^{13}+c_3e^{14}+c_4e^{23}+c_5e^{24}-e^{45}, \label{eq9'}\\ &de^4=f_2e^{13}+f_3e^{14}+f_4e^{23}+f_5e^{24}+e^{35},\nonumber \\
   &de^5=k_1e^{12}+k_6e^{34}.\nonumber 
\end{align}
First notice that $d\Phi=0$ implies
\begin{equation}\label{dphi=0}
a_6+c_4+f_5=0,\qquad -b_6+c_2+f_3=0.
\end{equation}
Imposing $d^2e^k=0, ~k=1,...,5$, one gets three systems of equations.

\begin{enumerate}
    \item First system:
   \begin{eqnarray*}
       f_2=-c_3,\quad f_3=c_2,\quad f_4=-c_5, \quad f_5=c_4, \label{eq10'}
   \end{eqnarray*}
   which, together with \eqref{dphi=0}, gives
   \begin{equation}\label{dphi=0A}
   a_6=-2c_4,\qquad b_6=2c_2.
   \end{equation}
  \item Second system:
   \begin{eqnarray*}
       a_6=-2c_4\frac{k_6}{k_1}, \quad  b_6=2c_2\frac{k_6}{k_1}. \label{eqk''}
   \end{eqnarray*}
   which, together with \eqref{dphi=0A} gives $k_1=k_6$. In particular $d\eta=k_1\Phi$, and 
 the structure equations (\ref{eq9'}) of $\mathfrak{g}$ can be rewritten as 
   \begin{align*}
 &de^1=a_1e^{12}-2c_4e^{34}, \\
  &de^2=b_1e^{12}+2c_2e^{34}, \\
 &de^3=c_2e^{13}+c_3e^{14}+c_4e^{23}-f_4e^{24}-e^{45}, \label{eqnk'}\\ &de^4=-c_3e^{13}+c_2e^{14}+f_4e^{23}+c_4e^{24}+e^{35}, \\
   &de^5=k_1e^{12}+k_1e^{34}.
\end{align*}  

   \item Third system, involving coefficients in the above structure equations:
 \[
   \begin{alignedat}{3} 
 c_2(a_1+2c_4) &= 0, \quad & c_4(a_1+2c_4) &=0, \quad & c_2(-b_1+2c_2) &= 0, \\
c_4(-b_1+2c_2) &= 0, \quad & a_1c_3-b_1f_4+k_1 &=0, \quad & c_2a_1+c_4b_1 &= 0. 
\end{alignedat}  
\]
\end{enumerate}
At this point we notice that, computing the tensor $N_\varphi$ one obtains $N_\varphi(e_i,e_j)=0$ for every $i,j=1,\ldots,5$. Therefore, the structure is normal, and being also $d\eta=k_1\Phi$, it is $\alpha$-Sasakian, with $k_1=2\alpha$. Notice that, a deformation of type \eqref{deformation} preserves the property for the structure to be $\alpha$-Sasakian, up to rescaling the coefficient $\alpha$. The solutions to the third system of equations are the following:
  \begin{align*}
 &\text{(A1)}~b_1=c_2=c_4=0,c_3\neq0,a_1=-\frac{k_1}{c_3}, \\
&\text{(A2)}~c_2=c_4=0,b_1\neq0,f_4=\frac{k_1+a_1c_3}{b_1}, \\
     &\text{(A3)}~b_1=c_2=0,a_1\neq0,c_3=-\frac{k_1}{a_1},c_4=-\frac{1}{2}a_1, \\
     &\text{(A4)}~b_1\neq0,c_2=\frac{1}{2}b_1,c_4=-\frac{1}{2}a_1,f_4=\frac{k_1+a_1c_3}{b_1}.
\end{align*}
For $(A1)$ and $(A2)$ we distinguish two cases according to the sign $k_1$. When $k_1 > 0$, 
\[
(A1), (A2) \simeq \mathfrak{aff}(\mathbb{R}) \times \mathfrak{sl}(2,\mathbb{R}),
\]
where the affine and semisimple components are given by
\[
\mathfrak{aff}(\mathbb{R}) =
\begin{cases}
\operatorname{Span}\{ f_4 e_1 + c_3 e_2,\; e_1 - c_3 e_5 \}, & \text{for (A1)}, \\[2mm]
\operatorname{Span}\{ a_1 e_1 + b_1 e_2 + k_1 e_5,\; e_1 - c_3 e_5 \}, & \text{for (A2)},
\end{cases}
\]
and
\[
\mathfrak{sl}(2,\mathbb{R}) = \operatorname{Span}\{ e_3,\, e_4,\, e_5 \}.
\]
When $k_1 < 0$, 
\[
(A1), (A2) \simeq \mathfrak{aff}(\mathbb{R}) \times \mathfrak{su}(2).
\]
where the affine parts remain the same, while the semisimple component changes to
\[
\mathfrak{su}(2) = \operatorname{Span}\{ e_3,\, e_4,\, e_5 \}.
\]
In the remaining cases, $(A3)$ and $(A4)$ take the form 
     $$\text{(A3),(A4)}\simeq\mathbb{R}^2\ltimes\mathfrak{h}_3.$$
  For $\text{(A3)}$, a convenient basis is
     \[
\left\{
E_1 = a_1 e_1 + k_1e_5,~
E_2 = \frac{1}{a_1} e_2,~
E_j = e_j,~ j = 3,4,5
\right\},
\]
with $\mathbb{R}^2 = \text{Span}\{E_2, E_5\}$, $\mathfrak{h}_3 = \text{Span}\{E_1, E_3, E_4\}$ and
\[
\text{ad}_{E_2} =
\begin{pmatrix}
1 & 0 & 0 \\
0 & \tfrac{1}{2} & \tfrac{f_4}{a_1} \\
0 & -\tfrac{f_4}{a_1} & \tfrac{1}{2}
\end{pmatrix}, \qquad
\text{ad}_{E_5} =
\begin{pmatrix}
0 & 0 & 0 \\
0 & 0 & -1 \\
0 & 1 & 0
\end{pmatrix}.
\]
For (A4), a similar basis can be chosen as 
\[
\left\{
F_1 = a_1 e_1 + b_1 e_2 + k_1e_5, ~
F_2 = \frac{1}{b_1} e_1, ~
F_j = e_j, ~ j = 3,4,5
\right\},
\]
with $\mathbb{R}^2 = \text{Span}\{F_2, F_5\}$, $\mathfrak{h}_3 = \text{Span}\{F_1, F_3, F_4\}$ and
\[
\text{ad}_{F_2} =
\begin{pmatrix}
-1 & 0 & 0 \\
0 & -\tfrac{1}{2} & -\tfrac{c_3}{b_1} \\
0 & \tfrac{c_3}{b_1} & -\tfrac{1}{2}
\end{pmatrix}, \qquad
\text{ad}_{F_5} =
\begin{pmatrix}
0 & 0 & 0 \\
0 & 0 & -1 \\
0 & 1 & 0
\end{pmatrix}.
\]

We now consider \textbf{Case B}. Using ($\ref{eqnk}$), we obtain
\begin{align}
 &de^1=a_1e^{12}+a_6e^{34},\nonumber \\
  &de^2=b_1e^{12}+b_6e^{34},\nonumber \\
 &de^3=c_2e^{13}+c_3e^{14}+c_4e^{23}+c_5e^{24}+e^{45}, \label{eq9}\\ &de^4=f_2e^{13}+f_3e^{14}+f_4e^{23}+f_5e^{24}-e^{35},\nonumber \\
   &de^5=k_1e^{12}+k_6e^{34}.\nonumber 
\end{align}    
Again in this case, $d\Phi=0$ gives the same equations as in \eqref{dphi=0},
and imposing $d^2e^k=0, ~k=1,...,5$, one gets three systems of equations.

\begin{enumerate}
\item First system:
   \begin{eqnarray*}
       f_2=-c_3,\quad f_3=c_2,\quad f_4=-c_5, \quad f_5=c_4, 
   \end{eqnarray*}
   which, together with \eqref{dphi=0}, give
   \begin{equation}\label{dphi=0B}
   a_6=-2c_4,\qquad b_6=2c_2.
   \end{equation}
\item Second system:
 \begin{eqnarray*}
       a_6=-2c_4\frac{k_6}{k_1}, \quad  f_3=b_6\frac{k_1}{2k_6}. \label{eqk}
   \end{eqnarray*}
   which together with \eqref{dphi=0B} give $k_1=k_6$. The structure equations (\ref{eq9}) of $\mathfrak{g}$ in this setting are
   \begin{align*}
 &de^1=a_1e^{12}-2c_4e^{34},\nonumber \\
  &de^2=b_1e^{12}+b_6e^{34},\nonumber \\
 &de^3=\frac{b_6}{2}e^{13}+c_3e^{14}+c_4e^{23}-f_4e^{24}+e^{45}, \label{eqnk'}\\ 
 &de^4=-c_3e^{13}+\frac{b_6}{2}e^{14}+f_4e^{23}+c_4e^{24}-e^{35},\nonumber \\
   &de^5=k_1e^{12}+k_1e^{34}.\nonumber 
\end{align*}
\item Third system:
\[    
   \begin{alignedat}{3} 
 c_4(a_1+2c_4) &= 0, \quad & b_6(b_6-b_1) &=0, \quad & c_4(b_1-b_6) &= 0, \\
b_6(a_1+2c_4) &= 0, \quad & 2b_1c_4+a_1b_6 &=0, \quad & c_3a_1-b_1f_4-k_1 &= 0. 
\end{alignedat}  
\] 
\end{enumerate}
Again, one can show that the structure is normal, and being $d\eta=k_1\Phi$, it is $\alpha$-Sasakian, with $k_1=2\alpha$.  
The solutions for the third system are the following:
  \begin{align*}
 &\text{(B1)}~b_1=b_6=c_4=0,c_3=\frac{k_1}{a_1}, \\
 &\text{(B2)}~b_6=c_4=0,f_4=\frac{a_1c_3-k_1}{b_1}, \\
     &\text{(B3)}~b_1=b_6=0,c_3=\frac{k_1}{a_1},c_4=-\frac{1}{2}a_1, \\
     &\text{(B4)}~b_6=b_1,c_4=-\frac{1}{2}a_1,f_4=\frac{a_1c_3-k_1}{b_1}.
\end{align*}
For the first two cases, $(B1)$ and $(B2)$ have the same general structure.  
The only difference between them arises from the sign of the constant $k_1$. When $k_1 > 0$, 
\[
(B1), (B2) \;\simeq\; \mathfrak{aff}(\mathbb{R}) \times \mathfrak{su}(2),
\]
where the affine and semisimple components are given by
\[
\mathfrak{aff}(\mathbb{R}) =
\begin{cases}
\operatorname{Span}\{a_1 e_1 + k_1e_5,\; e_2 - f_4 e_5 \}, & \text{for (B1)}, \\[2mm]
\operatorname{Span}\{ a_1 e_1 + b_1 e_2 + k_1 e_5,\; e_1 + c_3 e_5 \}, & \text{for (B2)},
\end{cases}
\]
and
\[
\mathfrak{su}(2) = \operatorname{Span}\{ e_3,\, e_4,\, e_5 \}.
\]
When $k_1 < 0$, the affine parts remain the same, while the semisimple component changes to
\[
\mathfrak{sl}(2,\mathbb{R}) = \operatorname{Span}\{ e_3,\, e_4,\, e_5 \},
\]
so that
\[
(B1), (B2) \;\simeq\; \mathfrak{aff}(\mathbb{R}) \times \mathfrak{sl}(2,\mathbb{R}).
\]
  In the remaining cases,  $(B3)$ and $(B4)$ take the form
     $${(B3),(B4)}\simeq\mathbb{R}^2\ltimes\mathfrak{h}_3$$
     For ${(B3)}$,  a convenient basis is
     \[
\left\{
G_1 = a_1 e_1 + k_1e_5,~
G_2 = \frac{1}{a_1} e_2,~
G_j = e_j,~ j = 3,4,5
\right\},
\]
with $\mathbb{R}^2 = \text{Span}\{G_2, G_5\}$, $\mathfrak{h}_3 = \text{Span}\{G_1, G_3, G_4\}$ and
\[
\text{ad}_{G_2} =
\begin{pmatrix}
1 & 0 & 0 \\
0 & \tfrac{1}{2} & \tfrac{f_4}{a_1} \\
0 & -\tfrac{f_4}{a_1} & \tfrac{1}{2}
\end{pmatrix}, \qquad
\text{ad}_{G_5} =
\begin{pmatrix}
0 & 0 & 0 \\
0 & 0 & 1 \\
0 & -1 & 0
\end{pmatrix}.
\]
For $(B4)$, a similar basis can be chosen as 
\[
\left\{
H_1 = a_1 e_1 + b_1 e_2 + k_1e_5, ~
H_2 = \frac{1}{b_1} e_1, ~
H_j = e_j, ~ j = 3,4,5
\right\},
\]
with $\mathbb{R}^2 = \text{Span}\{H_2, H_5\}$, $\mathfrak{h}_3 = \text{Span}\{H_1, H_3, H_4\}$ and
\[
\text{ad}_{H_2} =
\begin{pmatrix}
-1 & 0 & 0 \\
0 & -\tfrac{1}{2} & -\tfrac{c_3}{b_1} \\
0 & \tfrac{c_3}{b_1} & -\tfrac{1}{2}
\end{pmatrix}, \qquad
\text{ad}_{H_5} =
\begin{pmatrix}
0 & 0 & 0 \\
0 & 0 & 1 \\
0 & -1 & 0
\end{pmatrix}.
\]

In cases (A3), (A4), (B3), (B4), the corresponding Lie algebras are all isomorphic to a semidirect product $\mathfrak{g}_t=\mathbb{R}^2\ltimes_{\varphi_t}\mathfrak{h}_3$, where $\mathbb{R}^2=Span\{X,Y\}$, $\mathfrak{h}_3=Span\{v_1,v_2,v_3\}$ and    
 $$[v_1,v_3]=-v_1, \quad
\varphi_t(X)=
\begin{pmatrix}
1 & 0 & 0 \\
0 & \tfrac{1}{2} & t\\
0 & -t & -\frac{1}{2}
\end{pmatrix}, \quad
\varphi_t(Y)=
\begin{pmatrix}
0 & 0 & 0 \\
0 & 0 & -1 \\
0 & 1 & 0
\end{pmatrix}.
$$
By changing $X$ to $X + t Y$ we obtain an isomorphism between $\mathfrak{g}_t$ and $\mathfrak{g}_0$ with
structure equations 
\[
\begin{aligned}
  \left[e_1, e_3\right] &= e_3, &\quad [e_1, e_4] &= \tfrac{1}{2} e_4, &\quad [e_1, e_5] &= \tfrac{1}{2} e_5, \\
  [e_2, e_4] &= e_5, &\quad [e_2, e_5] &= -e_4, &\quad [e_4, e_5] &= -e_3.
\end{aligned}
\]
\end{proof}

We now consider the Ricci curvature of the above Lie algebras, clarifying whether they admit an $\eta$-Einstein structure, namely they satisfy  \eqref{eta-einstein}. In \cite[Section 4]{AF} it is showed that the Lie algebra $\mathfrak{su}(2) \times \mathfrak{aff}(\mathbb{R})$ does not admit any  $\eta$-Einstein Sasakian structure. Consequently, it cannot admit any $\eta$-Einstein transversely K\"ahler a.c.m. structure of maximal rank and with nontrivial center. Indeed, by Theorem \ref{theorem13}, such a structure would be necessary $\alpha$-Sasakian and an homothetic deformation would provide an $\eta$-Einstein Sasakian structure. On the other hand, the Lie algebras 
$\mathfrak{g}_0$ and $\mathfrak{sl}(2, \mathbb{R}) \times \mathfrak{aff}(\mathbb{R})$ 
 admit $\eta$-Einstein Sasakian structures which are never Einstein (\cite{AF}). Hence, we obtain the following:
 \begin{corollary}\label{corollary-eta-eistein}
     The only $5$-dimensional Lie algebras with trivial center admitting an $\eta$-Einstein transversely K\"ahler almost contact metric structure of maximal rank are  $\mathfrak{sl}(2, \mathbb{R}) \times \mathfrak{aff}(\mathbb{R})$ and $\mathfrak{g}_0\simeq \mathbb{R}^2 \ltimes\mathfrak{h}_3$, in which cases the structure is $\eta$-Einstein $\alpha$-Sasakian, non Einstein.
 \end{corollary}

\subsection{The case of nontrivial center in dimension \ensuremath{5} and the \ensuremath{\eta}-Einstein condition}\label{section-dim5-center}

According to Theorem \ref{theorem6}, $5$-dimensional transversely K\"ahler almost contact metric Lie algebras $(\mathfrak{g},\varphi,\xi,\eta,g)$ of maximal rank and with nontrivial center, are in one-to-one correspondence  with $4$-dimensional K\"ahler Lie algebras $(\mathfrak{h},J,h,\omega)$  endowed with a symplectic $2$-form $\omega$. Further, by Proposition \ref{proposition-isom}, two transversely K\"ahler almost contact metric Lie algebras, of maximal rank and with nontrivial center, are isomorphic if and only if their K\"ahler quotients admit an isomorphism which is also a symplectomorphism with respect the two symplectic forms realizing the extensions. 

In \cite{Ovandop}, Ovando classified $4$-dimensional pseudo-Kähler Lie algebras. Those admitting a positive-definite inner product (and hence a Kähler structure) are listed in Table \ref{table1}. See also \cite{CF}, where the different sign in the expression of the K\"ahler form is justified by the fact that we use a different convention, namely $\Omega(X,Y)=h(X,JY)$. On the other hand, Ovando also classified the symplectic Lie algebras of dimension $4$ (\cite{Ovando}). Table \ref{table2} contains all possible symplectic $2$-forms for the Lie algebras admitting a K\"ahler structure.

Building on these results, we employ Proposition \ref{proposition10} to compute the Ricci curvature of the central extensions of the $4$-dimensional Kähler Lie algebras reproduced in Table~\ref{table1} via the symplectic structures in Table~\ref{table2}. We then obtain the classification of $\eta$-Einstein $5$-dimensional transversely K\"ahler almost contact metric Lie algebras with maximal rank and nontrivial center, as in the following Theorem.

\renewcommand{\arraystretch}{1.3} 
\renewcommand{\tablename}{\textbf{Table}} 
\begin{table}[h!]
\small
\centering
\begin{tabular}[H]{lll}
\hline
{Lie algebra} &  {Complex structure} &  {K\"ahler form ($\Omega$)} \\ \hline \noalign{\hrule height 0.45mm}
 $\mathfrak{r}\mathfrak{r}_{3,0}=(0,-e^{12},0,0)$ & $Je_1=e_2, Je_3=e_4$ & $-(ae^{12}+be^{34}), a,b>0$ \\
$\mathfrak{r}\mathfrak{r}'_{3,0}=(0,-e^{13},e^{12},0)$& $Je_1=e_4, Je_2=e_3$ & $-(ae^{14}+be^{23}), a,b>0 $\\
$\mathfrak{r}_2\mathfrak{r}_2=(0,-e^{12},0,-e^{34})$ &$Je_1=e_2, Je_3=e_4$ & $-(ae^{12}+be^{34}), a,b>0$\\
$\mathfrak{r}'_{4,0,\delta}=(e^{14},\delta e^{34},-\delta e^{24},0)$ &$J_1e_4=e_1, J_1e_2=e_3$ & $-(ae^{14}+be^{23}), a<0,  b>0 $\\
&$J_2e_4=e_1, J_2e_3=e_2$ & $-(ae^{14}+be^{23}), a,b<0$\\
 $\mathfrak{d}_{4,2}=(2e^{14},-e^{24},-e^{12}+e^{34},0)$& $Je_1=e_4, Je_2=e_3$ & $-(ae^{14}+be^{23}), a,b>0 $\\
$\mathfrak{d}_{4,1/2}=(\frac{1}{2}e^{14},\frac{1}{2}e^{24},-e^{12}+e^{34},0)$& $Je_1=e_2, Je_4=e_3$ & $-a(e^{12}-e^{34}), a>0 $\\
 $\mathfrak{d}'_{4,\delta}=(\frac{\delta}{2}e^{14}+e^{24},-e^{14}+\frac{\delta}{2}e^{24},$\\~~~~~~~~~~$-e^{12}+\delta e^{34},0)$ &  $J_1e_1=e_2, J_1e_4=e_3$ &  $-a(e^{12}-\delta e^{34}), a>0 $\\
&  $J_3e_2=e_1, J_3e_3=e_4$ &  $-a(e^{12}-\delta e^{34}), a<0 $\\
\hline
\end{tabular}
\caption{ Four-dimensional non-abelian K\"ahler Lie algebras}
\label{table1}
\end{table}

\vspace{0.3cm}

\renewcommand{\arraystretch}{1.15}
\renewcommand{\tablename}{\textbf{Table}} 
\begin{table}[h!]
\centering
\begin{tabular}{ll}
\hline
{Lie algebra} & 
 {Symplectic forms ($\omega$)} \\ \hline \noalign{\hrule height 0.45mm}

{$\mathfrak{r}\mathfrak{r}_{3,0}$} &  

{$a_{12}e^{12}+a_{13}e^{13}+a_{14}e^{14}+a_{34}e^{34}, a_{12}a_{34}\neq0$} \\

{$\mathfrak{r}\mathfrak{r}'_{3,0}$} &  {$a_{12}e^{12}+a_{13}e^{13}+a_{14}e^{14}+a_{23}e^{23}, a_{14}a_{23}\neq0$} \\

{$\mathfrak{r}_2\mathfrak{r}_{2}$ }&  {$a_{12}e^{12}+a_{13}e^{13}+a_{34}e^{34}, a_{12}a_{34}\neq0$ }\\
{ $ \mathfrak{r}'_{4,0,\delta}$}&  {$a_{14}e^{14}+a_{23}e^{23}+a_{24}e^{24}+a_{34}e^{34}, a_{14}a_{23}\neq0, \delta \neq0$} \\
{ $ \mathfrak{d}_{4,2}$}&  {$a_{12-34}(e^{12}-e^{34})+a_{14}e^{14}+a_{23}e^{23}+a_{24}e^{24}, -a^2_{12-34}+a_{14}a_{23}\neq0, $} \\
{ $ \mathfrak{d}_{4,1/2}$}&  {$a_{12-34}(e^{12}-e^{34})+a_{14}e^{14}+a_{24}e^{24}, a_{12-34}\neq0$} \\
{ $ \mathfrak{d}'_{4,\delta}$}&  {$a_{12-\delta 34}(e^{12}-\delta e^{34})+a_{14}e^{14}+a_{24}e^{24}, a_{-12+\delta34}\neq0, \delta \neq 0 $ }\\
\hline
\end{tabular}
\caption{Symplectic forms on $4$-dimensional non-abelian K\"ahler Lie algebras}
\label{table2}
\end{table}

\begin{theorem} \label{theorem14}
     Let $(\mathfrak{g},\varphi,\xi,\eta,g)$ be a $5$-dimensional transversely K\"ahler a.c.m. Lie algebra of maximal rank with nontrivial center. 
     Then $(\mathfrak{g},\varphi,\xi,\eta,g)$ is {$\eta$-Einstein}  if and only if it is isomorphic to one of the Lie algebras listed in Table \ref{table-eta-einstein}. 
  In particular,
  \begin{enumerate}
  \item $\mathfrak{d}_{4,2}$ is the only $4$-dimensional K\"ahler Lie algebra which does not admit any $\eta$-Einstein extension via a symplectic form.
  \item In all the cases $\mathfrak{g}'_2,\dots,\mathfrak{g}'_7$, where the K\"ahler quotient is not abelian,  the almost contact metric structure is quasi Sasakian. 
  \item All, and only all, $4$-dimensional K\"ahler-Einstein Lie algebras admit  $\alpha$-Sasakian $\eta$-Einstein extensions.
  \item $\mathfrak{rr}'_{3,0}$ is the only $4$-dimensional non abelian K\"ahler-Einstein Lie algebra admitting $\eta$-Einstein extensions which are quasi Sasakian non $\alpha$-Sasakian.
  \item The Heisenberg Lie algebra $\mathfrak{h}_5$ is the only $5$-dimensional Lie algebra admitting $\eta$-Einstein structures which are not quasi-Sasakian (including anti-quasi-Sasakian structures).
  \item In any case the $\eta$-Einstein almost contact metric Lie algebra in not Einstein.
  \end{enumerate}
 \end{theorem}

\begin{proof}
For each K\"ahler Lie algebra listed in Table \ref{table1}, we consider central extensions via the symplectic forms in Table \ref{table2}. We then apply Proposition \ref{proposition10} to compute the Ricci curvature. 

We provide the details of the computations for  the central extension of the K\"ahler Lie algebra $\mathfrak{h}=(\mathfrak{r}\mathfrak{r}'_{3,0},J,\Omega)$. The Lie brackets of this Lie algebra are the following
$$[e_1,e_2]_{\mathfrak{h}}=-e_3,\quad [e_1,e_3]_{\mathfrak{h}}=e_2,\quad [e_i,e_j]_{\mathfrak{h}}=0\ \mbox{otherwise}. $$
Consider the positive definite inner product given by $h(X,Y)=\Omega(X,JY)$, which is  determined, with respect to $\lbrace e_1, ...,e_4 \rbrace$, by 
$$h_{11}=h_{44}=a, ~~~~~~h_{22}=h_{33}=b, ~~~~~h_{ij}=0 ~~~\text{if}~i \neq j.$$
An orthonormal basis of $(\mathfrak{r}\mathfrak{r}'_{3,0},h)$ is
$$\bar{e}_1=\frac{e_1}{\sqrt{a}}, ~~~~\bar{e}_2=\frac{e_2}{\sqrt{b}}, ~~~~\bar{e}_3=\frac{e_3}{\sqrt{b}}, ~~~~\bar{e}_4=\frac{e_4}{\sqrt{a}},$$
with $a,b>0$. Using the Kozsul formula
\[
2h(\nabla_{\bar{e}_i}{\bar{e}_j},{\bar{e}_k})=h([{\bar{e}_i},{\bar{e}_j}],{\bar{e}_k})-h([{\bar{e}_j},{\bar{e}_k}],{\bar{e}_i})+h([{\bar{e}_k},{\bar{e}_i}],{\bar{e}_j}),
\]
the Levi-Civita connection is described as follows:
 \begin{eqnarray*}
\nabla^\mathfrak{h}_{\bar{e}_1}\bar{e}_3=\frac{\bar{e}_2}{\sqrt{a}}, ~~~~~ \nabla^\mathfrak{h}_{\bar{e}_1}\bar{e}_2=-\frac{\bar{e}_3}{\sqrt{a}},~~~~~ \nabla^\mathfrak{h}_{\bar{e}_i}\bar{e}_j=0~~\mbox{otherwise}.
 \end{eqnarray*} 
Subsequently, for  the curvature of $(\mathfrak{r}\mathfrak{r}'_{3,0}, h)$  one obtains \[ R^\mathfrak{h}(\bar{e}_i,\bar{e}_j)\bar{e}_k=0,\]
so that $(\mathfrak{r}\mathfrak{r}'_{3,0}, h)$ is flat and the Ricci curvature curvature $\rho^{\mathfrak{h}}$ identically vanishes.

Now, as in Table \ref{table2}, we  consider the symplectic form on the Lie algebra $\mathfrak{rr}_{3,0}'$ given by $$\omega=a_{12}e^{12}+a_{13}e^{13}+a_{14}e^{14}+a_{23}e^{23},\ a_{14}a_{23}\neq 0,$$  for which we have
\begin{eqnarray}\nonumber
\omega(\bar{e}_1,\bar{e}_2)=\frac{a_{12}}{\sqrt{ab}}, ~~~~\omega(\bar{e}_1,\bar{e}_3)=\frac{a_{13}}{\sqrt{ab}},~~~~ \omega(\bar{e}_1,\bar{e}_4)=\frac{a_{14}}{a},\\ \label{eqw}
~~~ \omega(\bar{e}_2,\bar{e}_3)=\frac{a_{23}}{b},~~~\omega(\bar{e}_i,\bar{e}_j)=0, ~\mbox{for} ~i,j={1,..,4}.    
\end{eqnarray} 

Let $(\mathfrak{g},\varphi,\xi,\eta,g)$ be the almost contact metric Lie algebra obtained as  the central extension of the Lie algebra $(\mathfrak{h},J,h)$ via the simplectic form $\omega$. We compute the Ricci tensor of $\mathfrak{g}$. Firstly, using (\ref{cr.1}) and (\ref{eqw}), one can write 
\begin{align*}
    \rho(\bar{e}_1,\bar{e}_1) &= -\frac{1}{2}\Big(\frac{a_{12}^2}{ab}+\frac{a_{13}^2}{ab}+\frac{a_{14}^2}{a^2}\Big), &\qquad \qquad \rho(\bar{e}_1,\bar{e}_2)     
    &=-\frac{1}{2}\Big(\frac{a_{13}}{\sqrt{ab}} \cdot\frac{a_{23}}{b}\Big), \\
    \rho(\bar{e}_2,\bar{e}_2) &=-\frac{1}{2}\Big(\frac{a_{12}^2}{ab}+\frac{a_{23}^2}{b^2}\Big), &\qquad \qquad \rho(\bar{e}_1,\bar{e}_3) &=\frac{1}{2}\Big(\frac{a_{12}}{\sqrt{ab}} \cdot\frac{a_{23}}{b}\Big), \\
    \rho(\bar{e}_3,\bar{e}_3) &=-\frac{1}{2}\Big(\frac{a_{13}^2}{ab}+\frac{a_{23}^2}{b^2}\Big), &\qquad \qquad  \rho(\bar{e}_1,\bar{e}_4) &= 0, \\
    \rho(\bar{e}_4,\bar{e}_4) &= -\frac{1}{2}\Big(\frac{a_{14}^2}{a^2}+\frac{a_{13}^2}{ab}\Big), &\qquad \qquad  \rho(\bar{e}_2,\bar{e}_3) &= -\frac{1}{2}\Big(\frac{a_{12}}{\sqrt{ab}} \cdot\frac{a_{13}}{\sqrt{ab}}\Big), \\
    \rho(\bar{e}_3,\bar{e}_4) &= -\frac{1}{2}\Big(\frac{a_{13}}{\sqrt{ab}} \cdot\frac{a_{14}}{a}\Big), & \qquad \qquad \rho(\bar{e}_2,\bar{e}_4) &= -\frac{1}{2}\Big(\frac{a_{12}}{\sqrt{ab}} \cdot\frac{a_{14}}{a}\Big). 
\end{align*}
Using (\ref{cr.2}), we get
\begin{eqnarray*}
\rho(\xi,\bar{e}_1)=\rho(\xi,\bar{e}_4)=0, \qquad
\rho(\xi,\bar{e}_2)
=- \frac{a_{13}}{2a\sqrt{b}}, \qquad
\rho(\xi,\bar{e}_3)
    = \frac{a_{12}}{2a\sqrt{b}}, 
\end{eqnarray*}
and by (\ref{cr.3}), we obtain
\begin{eqnarray*}
    \rho(\xi,\xi)
    &=&\frac{1}{2}  \Big(\frac{a_{12}^2+a_{13}^2}{ab}+\frac{a_{14}^2}{a^2} +\frac{a_{23}^2}{b^2} \Big).
\end{eqnarray*}

The structure is $\eta$-Einstein if and only if $\rho(\bar e_i,\bar e_i)=\rho(\bar e_j,\bar e_j)$ and $\rho(\bar e_i,\bar e_j)=\rho(\xi,\bar e_j)=0$ for $i\ne j$, which are equivalent to  \[a_{12}=a_{13}=0,\qquad \frac{a_{14}^2}{a^2}=\frac{a_{23}^2}{b^2}.\] 
If we set $\lambda=\frac{a_{14}}{a}=\pm\frac{a_{23}}{b}\ne0$, we have the following possibilities:
\[\omega=\lambda(ae^{14}+be^{23})=\lambda\Omega,\]
\[\omega=\lambda(ae^{14}-be^{23}).\]
Therefore, for the almost contact metric structure $(\varphi,\xi,\eta,g)$, in the first case one has $d\eta=\lambda\Phi$, so that the structure is $\alpha$-Sasakian with $2\alpha=\lambda$. In the second case we can observe that the structure is quasi-Sasakian. Indeed, one can easily see that the symplectic form $\omega=\lambda(ae^{14}-be^{23})$ satisfies $\omega(JX,JY)=\omega(X,Y)$ for every $X,Y\in\mathfrak{h}$, and thus the structure is quasi Sasakian (see Remark \ref{remark aqS qS}). In both cases the expression for the Ricci curvature is the following:
\[
\mathrm{Ric}
= -\frac{a_{14}^{2}}{2a^{2}}\, g
+ \frac{3a_{14}^{2}}{2a^{2}}\, \eta \otimes \eta= -\frac{\lambda^{2}}{2} g
+ \frac{3\lambda^{2}}{2}\, \eta \otimes \eta,
\]
in particular, the Lie algebra is not Einstein.

The Ricci curvatures for the extensions of the remaining non abelian K\"ahler Lie algebras can be computed in an analogous manner. We only summarize in Table \ref{table4} the Ricci curvatures of all non abelian K\"ahler Lie algebras. In general, as consequence of Proposition \ref{proposition-isom}, the isomorphism classes of the almost contact metric Lie algebras depend on the parameters $a,b,\lambda,\mu,\delta$, where in particular $a,b,\delta$ determine the isomorphism classes of the K\"ahler quotients.

As regards the central extensions of $\mathbb{R}^4$, one can argue similarly. Consider the K\"ahler structure $(J,\Omega)$ defined by
\[Je_1=e_2,\quad Je_3=e_4,\quad \Omega=-(a e^{12}+be^{34}),\ a,b>0,\]
and the symplectic form
\[\omega=a_{12}e^{12}+a_{13}e^{13}+a_{14}e^{14}+a_{23}e^{23}+a_{24}e^{24}+a_{34}e^{34}, \quad a_{12}a_{34}-a_{13}a_{24}+a_{14}a_{23}\ne 0.\]
The $5$-dimensional Lie algebra $(\mathfrak{g},\varphi,\xi,\eta,g)$ obtained as central extension of $\mathbb{R}^4$ via $\omega$ is isomorphic to the Heisenberg Lie algebra $\mathfrak{h}_5$. Using Proposition \ref{proposition10}, one can show that $\mathfrak{g}$ is $\eta$-Einstein if and only if the symplectic form is
\[\omega=\lambda(ae^{12}\pm be^{34})+\beta(e^{13}\mp e^{24})+\gamma(e^{14}\pm e^{23}), \quad (\lambda,\beta,\gamma)\ne (0,0,0)\]
in which case the Ricci curvature tensor is given by $\mathrm{Ric}
=\left(\lambda^2+\frac{\beta^2+\gamma^2}{ab}\right)\left(-\frac{1}{2}g
+\frac{3}{2}\eta\otimes\eta\right)$. We show that it is not restrictive taking $\gamma=0$. Indeed, assuming
\[\omega=\lambda(ae^{12}+ be^{34})+\beta(e^{13}-e^{24})+\gamma(e^{14}+ e^{23}),\]
one can consider the basis
\begin{align*}
f_1&=\cos{\theta}\ e_1+\sin{\theta}\ e_2 &\quad f_3&=\cos{\theta}\ e_3+\sin{\theta}\ e_4\\
f_2&=-\sin{\theta}\ e_1+\cos{\theta}\ e_2 &\quad f_4&=-\sin{\theta}\ e_3+\cos{\theta}\ e_4
\end{align*}
with respect to which $f^{12}=e^{12}$, $f^{34}=e^{34}$, and 
\begin{align*}
f^{13}-f^{24}&=\cos{2\theta}\ (e^{13}-e^{24})+\sin{2\theta}\ (e^{14}+e^{23}),\\
f^{23}+f^{14}&=-\sin{2\theta}\ (e^{13}-e^{24})+\cos{2\theta}\ (e^{14}+e^{23}).
\end{align*}
Now, if $\beta=0$, we can take $\cos{2\theta}=0$ and $\sin{2\theta}=1$, so that $\omega=\lambda(af^{12}+ bf^{34})+\gamma(f^{13}- f^{24})$. If $\beta\ne0$, we can take $\tan{2\theta}=\frac{\gamma}{\beta}$ and straightforward computations give $\omega=\lambda(af^{12}+ bf^{34})\pm\sqrt{\beta^2+\gamma^2}(f^{13}- f^{24})$. 
In the case where
\[\omega=\lambda(ae^{12}-be^{34})+\beta(e^{13}+e^{24})+\gamma(e^{14}- e^{23}),\]
one can consider the basis
\begin{align*}
f_1&=\cos{\theta}\ e_1+\sin{\theta}\ e_2 &\quad f_3&=\cos{\theta}\ e_3-\sin{\theta}\ e_4\\
f_2&=-\sin{\theta}\ e_1+\cos{\theta}\ e_2 &\quad f_4&=\sin{\theta}\ e_3+\cos{\theta}\ e_4
\end{align*}
and argue analogously as in the previous case.

Finally, all claims 1. to 6. are consequences of the information collected in Table \ref{table-eta-einstein}.
\renewcommand{\arraystretch}{1.15} 
\renewcommand{\tablename}{\textbf{Table}} 
\begin{table}[h!]
\centering
\begin{tabular}{lll}
\hline
{Lie algebra} & 
 {orthonormal basis} & $\operatorname{Ric}$ \\ \hline \noalign{\hrule height 0.45mm}

{$\mathfrak{r}\mathfrak{r}_{3,0}$} &  
{$\frac{e_1}{\sqrt{a}},\frac{e_2}{\sqrt{a}},\frac{e_3}{\sqrt{b}},\frac{e_4}{\sqrt{b}}$} &$\operatorname{Ric}=\operatorname{diag}\left(-\frac1a,-\frac1a,0,0\right)$\\

{$\mathfrak{r}\mathfrak{r}'_{3,0}$} &  
{$\frac{e_1}{\sqrt{a}},\frac{e_2}{\sqrt{b}},\frac{e_3}{\sqrt{b}},\frac{e_4}{\sqrt{a}}$} &$\operatorname{Ric}=0$\\

{$\mathfrak{r}_2\mathfrak{r}_{2}$ }&   
{$\frac{e_1}{\sqrt{a}},\frac{e_2}{\sqrt{a}},\frac{e_3}{\sqrt{b}},\frac{e_4}{\sqrt{b}}$} &$\operatorname{Ric}=\operatorname{diag}\left(-\frac1a,-\frac1a,-\frac1b,-\frac1b\right)$\\

{$\mathfrak{r}'_{4,0,\delta}$ with $J_1$} &  
{$\frac{e_1}{\sqrt{-a}},\frac{e_2}{\sqrt{b}},\frac{e_3}{\sqrt{b}},\frac{e_4}{\sqrt{-a}}$} &$\operatorname{Ric}=\operatorname{diag}\left(\frac1a,0,0,\frac1a\right)$\\

{$\mathfrak{r}'_{4,0,\delta}$ with $J_2$} &  
{$\frac{e_1}{\sqrt{-a}},\frac{e_2}{\sqrt{-b}},\frac{e_3}{\sqrt{-b}},\frac{e_4}{\sqrt{-a}}$} &$\operatorname{Ric}=\operatorname{diag}\left(\frac1a,0,0,\frac1a\right)$\\

{ $ \mathfrak{d}_{4,2}$}&  
{$\frac{e_1}{\sqrt{a}},\frac{e_2}{\sqrt{b}},\frac{e_3}{\sqrt{b}},\frac{e_4}{\sqrt{a}}$} &$\operatorname{Ric}=\operatorname{diag}\left(-\frac{9}{2a},\frac{3}{2a},-\frac{3}{2a},-\frac{6}{a}\right)$\\

{ $\mathfrak{d}_{4,1/2}$}&   
{$\frac{e_1}{\sqrt{a}},\frac{e_2}{\sqrt{a}},\frac{e_3}{\sqrt{a}},\frac{e_4}{\sqrt{a}}$} &$\operatorname{Ric}=-\frac{3}{2a}\, h$\\

{$\mathfrak{\delta}'_{4,\delta}$ with $J_1$} &  
{$\frac{e_1}{\sqrt{a}},\frac{e_2}{\sqrt{a}},\frac{e_3}{\sqrt{a\delta}},\frac{e_4}{\sqrt{a\delta}}$} &$\operatorname{Ric}=-\frac{3\delta}{2a}\, h$\\

{$\mathfrak{\delta}'_{4,\delta}$ with $J_3$} &  
{$\frac{e_1}{\sqrt{-a}},\frac{e_2}{\sqrt{-a}},\frac{e_3}{\sqrt{-a\delta}},\frac{e_4}{\sqrt{-a\delta}}$} &$\operatorname{Ric}=\frac{3\delta}{2a}\,h$\\
\hline
\end{tabular}
\caption{Ricci curvature of non-abelian K\"ahler Lie algebras}
\label{table4}
\end{table}

\normalsize
\end{proof}

We can also discuss unimodularity, as follows:
\begin{proposition} 
    Let $(\mathfrak{g},\varphi,\xi,\eta,g)$ be a $5$-dimensional $\eta$-Einstein  transversely K\"ahler almost contact metric Lie algebra, with maximal rank and nontrivial center. If $\mathfrak{g}$ is unimodular, then $\mathfrak{g}$ is null $\eta$-Einstein, and thus isomorphic either to $\mathfrak{g}'_1\simeq \mathfrak{h}_5$ or to $\mathfrak{g}'_3 \simeq \mathbb{R}\ltimes(\mathfrak{h}_3\times\mathbb{R})$.
\end{proposition}
\begin{proof}
    Denoting by $(\mathfrak{h},J,h)$ the K\"ahler quotient, $\mathfrak{g}$ is unimodular if and only if so is $\mathfrak{h}$. On the other hand, the K\"ahler Lie algebra $\mathfrak{h}$ is unimodular iff it is flat (\cite{Hano}). Therefore $\mathfrak{g}$ is the central extension either of $\mathbb{R}^4$ of $\mathfrak{rr}'_{3,0}$, and the result follows from Theorem \ref{theorem14}.
\end{proof}

    Theorem \ref{theorem14} classifies, in particular, $5$-dimensional $\eta$-Einstein  \emph{quasi Sasakian} Lie algebras of maximal rank and with nontrivial center, and also $\eta$-Einstein \emph{$\alpha$-Sasakian} Lie algebras with nontrivial center. Point 4. in the statement of Theorem \ref{theorem14} is coherent with \cite[Proposition 3.3]{AC} where it is showed that a Sasakian Lie algebra with center is $\eta$-Einstein if and only if the K\"ahler quotient is Einstein. Consequently, Theorem \ref{theorem14} completes the classification of $5$-dimensional $\eta$-Einstein Sasakian Lie algebras with center contained in \cite{AF}. Referring to the list in Theorem \ref{theo-Sasaki-trivial}, we have the following:

\begin{theorem}
    The $5$-dimensional $\eta$-Einstein Sasakian Lie algebras with nontrivial center are $\mathfrak{g}_1\simeq \mathfrak{h}_5$, $\mathfrak{g}_3\simeq\mathbb{R}\ltimes (\mathfrak{h}_3\times \mathbb{R})$, $\mathfrak{g}_4\simeq\mathfrak{aff}(\mathbb{R})\times \mathfrak{aff}(\mathbb{R})\times\mathbb{R}$, $\mathfrak{g}_5\simeq\mathbb{R}\times (\mathbb{R}\ltimes \mathfrak{h}_3)$, and $\mathfrak{g}_7^\delta \simeq\mathbb{R}\times (\mathbb{R} \ltimes \mathfrak{h}_3),\delta>0$.
\end{theorem}

To conclude, Corollary \ref{corollary-eta-eistein} and Theorem \ref{theorem14} together lead us to the following result.

\begin{theorem}\label{theorem15}
   Any $5$-dimensional  Lie algebra admitting $\eta$-Einstein  transversely K\"ahler almost contact metric structure of maximal rank is isomorphic to one of the following: $\mathfrak{sl}(2,\mathbb{R})\times \mathfrak{aff}(\mathbb{R})$,  $\mathfrak{g}_0\simeq \mathbb{R}^2 \ltimes\mathfrak{h}_3$, $\mathfrak{g}'_1\simeq\mathfrak{h}_5$, $\mathfrak{g}'_2\simeq\mathfrak{aff}(\mathbb{R})\times \mathfrak{h}_3$, $\mathfrak{g}'_3\simeq\mathbb{R}\ltimes (\mathfrak{h}_3\times \mathbb{R})$, 
   $\mathfrak{g}'_4\simeq \mathfrak{aff}(\mathbb{R})\times \mathfrak{aff}(\mathbb{R})\times\mathbb{R}$, $\mathfrak{g}'_5\simeq\mathbb{R}\ltimes(\mathfrak{h}_3\times \mathbb{R})$, $\mathfrak{g}'_6\simeq\mathbb{R}\times (\mathbb{R}\ltimes \mathfrak{h}_3)$, $\mathfrak{g}'_7\simeq\mathbb{R}\times (\mathbb{R} \ltimes \mathfrak{h}_3)$.
\end{theorem}

\begin{landscape}
\setcellgapes{8pt} 
\makegapedcells
\begin{table}[htbp]
\centering
\resizebox{1.6\textwidth}{!}{
\begin{tabular}{|c|c|c|c|c|c|c|c|}
\hline
\multicolumn{2}{|c|}{\textbf{Extended Lie algebra}}
& \textbf{Structure tensor $\varphi$} 
&\textbf{\makecell{Fundamental \\ form $\Phi$}}
& \textbf{quasi Sasakian} 
& \textbf{$\alpha$-Sasakian}
& \textbf{\makecell{Transverse K\"ahler\\ Lie algebra}} 
& \textbf{Ricci curvature} 
\\
\hline
\multirow{2}{*}{\makecell{$\mathfrak{g}'_1 \simeq \mathfrak{h}_5$}}
&
\makecell{$\left(0,0,0,0,\lambda(ae^{12}+ be^{34})+\beta(e^{13}-e^{24})\right)$ \\[8pt] $a,b>0,\ (\lambda,\beta)\ne (0,0) $}
&
$\varphi e_1=e_2,\;
\varphi e_3=e_4,\;
\varphi \xi=0$
&
$-(ae^{12}+be^{34})$
&
\makecell{ iff $\beta=0$\\[8pt]
(aqS iff $\lambda=0$)}
&
iff $\beta=0$
&
$\mathbb{R}^4$ (Flat)
&
$\mathrm{Ric}
=\left(\lambda^2+\frac{\beta^2}{ab}\right)\left(-\frac{1}{2}g
+\frac{3}{2}\eta\otimes\eta\right)$
\\
\cline{2-8}
&
\makecell{$\left(0,0,0,0,\lambda(ae^{12}-be^{34})+\beta(e^{13}+ e^{24})\right)$\\[8pt]  $a,b>0,\ (\lambda,\beta)\ne (0,0)$}
&
$\varphi e_1=e_2,\;
\varphi e_3=e_4,\;
\varphi \xi=0$
&
$-(ae^{12}+be^{34})$
&
yes
&
no
&
$\mathbb{R}^4$ (Flat)
&
$\mathrm{Ric}
=\left(\lambda^2+\frac{\beta^2}{ab}\right)\left(-\frac{1}{2}g
+\frac{3}{2}\eta\otimes\eta\right)$
\\
\hline
$ \mathfrak{g}'_2 \simeq \mathfrak{aff}(\mathbb{R})\times \mathfrak{h}_3$ 
&
\makecell{$\left(0, -e^{12},0,0,\lambda a e^{12}+\mu b e^{34}\right)$\\[8pt]$ a,b>0,\ \lambda,\mu\ne0,\ \frac{\lambda^2}{2}+\frac1a=\frac{\mu^2}{2}$}
&
$
\varphi e_1=e_2,\;
\varphi e_3=e_4,\;
\varphi \xi=0
$
&
$-(ae^{12}+be^{34})$
&
yes
&
no
&
$\mathfrak{r}\mathfrak{r}_{3,0}$ (not Einstein)
&
$
\mathrm{Ric}
=-\frac{\mu^2}{2}g
+\left(\frac{3\mu^2}{2}-\frac1a\right)\eta\otimes\eta
$
\\
\hline
$\mathfrak{g}'_3 \simeq \mathbb{R}\ltimes(\mathfrak{h}_3\times\mathbb{R})$
&
\makecell{$\left(0,-e^{13},e^{12},0,\lambda(ae^{14}\pm be^{23})\right)$\\[8pt] $a,b>0,\ \lambda\ne 0$}
&
$
\varphi e_1=e_4,\;
\varphi e_2=e_3,\;
\varphi \xi=0
$
&
$-(ae^{14}+be^{23})$
&
yes
&
iff $de^5=\lambda(ae^{14}+be^{23})$
&
$\mathfrak{rr'}_{3,0}$ (Flat)
&
$
\mathrm{Ric}
=-\frac{\lambda^2}{2}g
+\frac{3\lambda^2}{2}\eta\otimes\eta
$
\\
\hline
$\mathfrak{g}'_4 \simeq \mathrm{aff}(\mathbb{R})\times
\mathrm{aff}(\mathbb{R})\times\mathbb{R}$
&
\makecell{$\left(0,-e^{12},0,-e^{34},\lambda ae^{12}+\mu b e^{34}\right)$\\[8pt] $a,b >0,\  \lambda,\mu\ne 0,\ \frac{\lambda^2}{2}+\frac1a=\frac{\mu^2}{2}+\frac{1}{b} $}
&
$
\varphi e_1=e_2,\;
\varphi e_3=e_4,\;
\varphi \xi=0
$
&
$-(ae^{12}+b e^{34})$
&
yes
&
iff $\lambda=\mu$ and $a=b$
&
$\mathfrak{r}_2\mathfrak{r}_2$ (Einstein iff $a=b$)
&
$
\mathrm{Ric}
=\Big(-\frac{\lambda^2}{2}-\frac{1}{a}\Big)g
+\Big(\lambda^2+\frac{\mu^2}{2}+\frac{1}{a}\Big)\eta\otimes\eta
$\\
\hline
\multirow{2}{*}{\makecell{$\mathfrak{g}'_5\simeq  \mathbb{R}\ltimes(\mathfrak{h}_3\times \mathbb{R})$}}
&
\makecell{$\left(e^{14},\delta e^{34},-\delta e^{24},0, \lambda ae^{14}+\mu b e^{23}\right)$\\[8pt] $a<0,\ b >0,\ \delta>0,\  \lambda,\mu\ne 0,\ \frac{\lambda^2}{2}-\frac1a=\frac{\mu^2}{2} $}
&
$
\varphi e_4=e_1,\;
\varphi e_2=e_3,\;
\varphi \xi=0
$
&
$-(ae^{14}+b e^{23})$
&
yes
&
no
&
$\mathfrak{r}'_{4,0,\delta}$ with $J_1$ (not Einstein)
&$
\mathrm{Ric}=
-\frac{\mu^2}{2}g
+\Big(\frac{3\mu^2}{2}+\frac{1}{a}\Big)\eta\otimes\eta
$\\
\cline{2-8}
&
\makecell{$\left(e^{14},\delta e^{34},-\delta e^{24},0, \lambda ae^{14}+\mu b e^{23}\right)$\\[8pt] $\quad a<0,\ b <0,\ \delta>0,\ \lambda,\mu\ne 0,\ \frac{\lambda^2}{2}-\frac1a=\frac{\mu^2}{2} $}
&
$
\varphi e_4=e_1,\;
\varphi e_3=e_2,\;
\varphi \xi=0
$
&
$-(ae^{14}+b e^{23})$
&
yes
&
no
&
$\mathfrak{r}'_{4,0,\delta}$ with $J_2$ (not Einstein)
&$
\mathrm{Ric}=
-\frac{\mu^2}{2}g
+\Big(\frac{3\mu^2}{2}+\frac{1}{a}\Big)\eta\otimes\eta
$\\
\hline
$\mathfrak{g}'_6 \simeq \mathbb{R}\times(\mathbb{R}\ltimes\mathfrak{h}_3)$
&
\makecell{$\left(\tfrac{1}{2}e^{14},\tfrac{1}{2}e^{24},
-e^{12}+e^{34},0,
\lambda a(e^{12}-e^{34})\right)$\\[8pt] $a>0,\ \lambda\ne 0$}
&
$\varphi e_1=e_2,\;
\varphi e_4=e_3,\;
\varphi \xi=0$
&
$-a(e^{12}-e^{34})$
&
yes
&
yes
&
$\mathfrak{d}_{4,\frac12}$ (Einstein)
&
$\mathrm{Ric}
=\Big(-\frac{\lambda^2}{2}-\frac{3}{2a}\Big)g
+\Big(
\frac{3\lambda^2}{2}+\frac{3}{2a}\Big)\eta\otimes\eta$
\\
\hline
\multirow{2}{*}{\makecell{$\mathfrak{g}'_7 \simeq \mathbb{R}\times(\mathbb{R}\ltimes\mathfrak{h}_3)$}}
&
\makecell{$\left(\tfrac{\delta}{2}e^{14}+e^{24},
-e^{14}+\tfrac{\delta}{2}e^{24},
-e^{12}+\delta e^{34},0,
\lambda a(e^{12}-\delta e^{34})\right)$\\[8pt] $a>0,\ \delta >0,\ \lambda\ne0$}
&
$\varphi e_1=e_2,\;
\varphi e_4=e_3,\;
\varphi \xi=0$
&
$-a(e^{12}-\delta e^{34})$
&
yes
&
yes
&
$\mathfrak{d}'_{4, \delta}$ with $J_1$ (Einstein)
&
$\mathrm{Ric}
=\Big(-\frac{\lambda^2}{2}-\frac{3\delta}{2a}\Big)g
+\Big(\frac{3\lambda^2}{2}+\frac{3\delta}{2a}
\Big)\eta\otimes\eta$
\\
\cline{2-8}
&
\makecell{$\left(\tfrac{\delta}{2}e^{14}+e^{24},
-e^{14}+\tfrac{\delta}{2}e^{24},
-e^{12}+\delta e^{34},0,
\lambda a(e^{12}-\delta e^{34})\right)$\\[8pt] $a<0,\ \delta >0,\ \lambda\ne0$ }
&
$\varphi e_2=e_1,\;
\varphi e_3=e_4,\;
\varphi \xi=0$
&
$-a(e^{12}-\delta e^{34})$
&
yes
&
yes
&
$\mathfrak{d}'_{4, \delta}$ with $J_3$ (Einstein)
&
$\mathrm{Ric}
=\Big(-\frac{\lambda^2}{2}+\frac{3\delta}{2a}\Big)g
+\Big(
\frac{3\lambda^2}{2}-\frac{3\delta}{2a}\Big)\eta\otimes\eta$
\\
\hline
\end{tabular}}
\caption{$\eta$-Einstein transversely K\"ahler almost contact metric Lie algebras of maximal rank with nontrivial center}\label{table-eta-einstein}
\end{table}
\end{landscape}

\bigskip

%%%%%%%%%%% 

\bigskip
\bigskip

\noindent{\sc Giulia Dileo}\\
Dipartimento di Matematica, Università degli Studi di Bari Aldo Moro, Via E. Orabona 4, 70125 Bari, Italy\\
\texttt{giulia.dileo@uniba.it}\\

\noindent{\sc Deniz Poyraz}\\
Department of Mathematics,
Ege University,
Izmir,
Turkiye
\\
\texttt{deniz.poyraz@ege.edu.tr}\\

\noindent{\sc Bayram \c{S}ahin}\\
Department of Mathematics,
Ege University,
Izmir,
Turkiye
\\
\texttt{bayram.sahin@ege.edu.tr}\\

\end{document}